\documentclass[11pt]{amsart}
\usepackage{import}
\usepackage{mathabx}
\usepackage{bbm}
\usepackage[T1]{fontenc}
\usepackage{amsfonts}
\usepackage{braket}
\usepackage{mathrsfs}
\usepackage{stmaryrd}
\usepackage{upgreek}
\usepackage{textcomp}

\usepackage{extarrow}
\usepackage{enumerate}
\usepackage{enumitem}
\usepackage{graphicx}
\usepackage[colorlinks=true, linkcolor=blue]{hyperref}
\usepackage{mathtools}
\usepackage{mathtext}
\usepackage{amsmath}
\usepackage{amssymb}
\usepackage{amsthm}
\usepackage{tikz}
\usepackage[all,cmtip]{xy}
\usepackage[a4paper,margin=1.2in]{geometry}

\usetikzlibrary{arrows}
\usetikzlibrary{matrix}
\usetikzlibrary{shapes}
\usetikzlibrary{snakes}
\usetikzlibrary{matrix}

\DeclareMathOperator{\Coker}{\textup{Coker}}

\DeclareMathOperator{\GL}{\textup{GL}}

\DeclareMathOperator{\Spec}{\textup{Spec}}

\DeclareMathOperator{\Tor}{\textup{Tor}}

\DeclareMathOperator{\et}{\textup{et}}

\DeclareMathOperator{\pr}{\textup{pr}}

\usepackage[citestyle=alphabetic, bibstyle=authortitle, backend=biber, giveninits=true, uniquelist = false, uniquename=init, isbn=false, maxcitenames=3, dashed=false, maxbibnames=999, doi=false, url=false]{biblatex}
\theoremstyle{plain}
\newtheorem{thm}{Theorem}[section]
\newtheorem*{thm*}{Theorem}
\newtheorem{lem}[thm]{Lemma}
\newtheorem{cor}[thm]{Corollary}
\newtheorem{prop}[thm]{Proposition}

\theoremstyle{remark}
\newtheorem{rmk}[thm]{Remark}

\newtheorem*{rmk*}{Remark}

\theoremstyle{definition}
\newtheorem{defn}[thm]{Definition} 

\newtheorem*{const*}{Construction}

\theoremstyle{plain}
\newtheorem{thmI}{Theorem}
\newtheorem{thmII}{Theorem}

\newcounter{NC}

\def \Frob {\mathrm{\bf Frob}}

\def \gr {{\mathrm{gr}}}
\def \EFin {{\mathrm{EFin}}}
\def \EFinet {{\mathrm{EFin}_{\textup{ét}}}}
\def \univ {{\rm univ}}

\begin{document}
\title[étale vs stratified]{The étale fundamental group and $F$-divided sheaves in characteristic $p>0$}
\author{Xiaotao Sun, Lei Zhang}

\address{ 
Xiaotao SUN\\
     Tianjin University\\
    School of Mathematics\\    
    Tianjin\\ China }
\email{xiaotaosun@tju.edu.cn}

 \address{Lei ZHANG\\
    Sun Yat-Sen University\\
    School of Mathematics (Zhuhai)\\    
    Zhuhai, 
    Guangdong Province\\ China}
\email{cumt559@gmail.com} 

\date{\today}

\global\long\def\A{\mathbb{A}}

\global\long\def\Ab{(\textup{Ab})}

\global\long\def\C{\mathbb{C}}

\global\long\def\Cat{(\textup{Cat})}

\global\long\def\Di#1{\textup{D}(#1)}

\global\long\def\E{\mathbb{E}}

\global\long\def\F{\mathbb{F}}

\global\long\def\GCov{G\textup{-Cov}}

\global\long\def\Gcat{(\textup{Galois cat})}

\global\long\def\Gfsets#1{#1\textup{-fsets}}

\global\long\def\Gm{\mathbb{G}_{m}}

\global\long\def\GrCov#1{\textup{D}(#1)\textup{-Cov}}

\global\long\def\Grp{(\textup{Grps})}

\global\long\def\Gsets#1{(#1\textup{-sets})}

\global\long\def\HCov{H\textup{-Cov}}

\global\long\def\MCov{\textup{D}(M)\textup{-Cov}}

\global\long\def\MHilb{M\textup{-Hilb}}

\global\long\def\N{\mathbb{N}}

\global\long\def\PGor{\textup{PGor}}

\global\long\def\PGrp{(\textup{Profinite Grp})}

\global\long\def\PP{\mathbb{P}}

\global\long\def\Pj{\mathbb{P}}

\global\long\def\Q{\mathbb{Q}}

\global\long\def\RCov#1{#1\textup{-Cov}}

\global\long\def\RR{\mathbb{R}}

\global\long\def\Sch{\textup{Sch}}

\global\long\def\WW{\textup{W}}

\global\long\def\Z{\mathbb{Z}}

\global\long\def\acts{\curvearrowright}

\global\long\def\alA{\mathscr{A}}

\global\long\def\alB{\mathscr{B}}

\global\long\def\arr{\longrightarrow}

\global\long\def\arrdi#1{\xlongrightarrow{#1}}

\global\long\def\catC{\mathscr{C}}

\global\long\def\catD{\mathscr{D}}

\global\long\def\catF{\mathscr{F}}

\global\long\def\catG{\mathscr{G}}

\global\long\def\comma{,\ }

\global\long\def\covU{\mathcal{U}}

\global\long\def\covV{\mathcal{V}}

\global\long\def\covW{\mathcal{W}}

\global\long\def\duale#1{{#1}^{\vee}}

\global\long\def\fasc#1{\widetilde{#1}}

\global\long\def\fsets{(\textup{f-sets})}

\global\long\def\iL{r\mathscr{L}}

\global\long\def\id{\textup{id}}

\global\long\def\la{\langle}

\global\long\def\odi#1{\mathcal{O}_{#1}}

\global\long\def\ra{\rangle}

\global\long\def\set{(\textup{Sets})}

\global\long\def\sets{(\textup{Sets})}

\global\long\def\shA{\mathcal{A}}

\global\long\def\shB{\mathcal{B}}

\global\long\def\shC{\mathcal{C}}

\global\long\def\shD{\mathcal{D}}

\global\long\def\shE{\mathcal{E}}

\global\long\def\shF{\mathcal{F}}

\global\long\def\shG{\mathcal{G}}

\global\long\def\shH{\mathcal{H}}

\global\long\def\shI{\mathcal{I}}

\global\long\def\shJ{\mathcal{J}}

\global\long\def\shK{\mathcal{K}}

\global\long\def\shL{\mathcal{L}}

\global\long\def\shM{\mathcal{M}}

\global\long\def\shN{\mathcal{N}}

\global\long\def\shO{\mathcal{O}}

\global\long\def\shP{\mathcal{P}}

\global\long\def\shQ{\mathcal{Q}}

\global\long\def\shR{\mathcal{R}}

\global\long\def\shS{\mathcal{S}}

\global\long\def\shT{\mathcal{T}}

\global\long\def\shU{\mathcal{U}}

\global\long\def\shV{\mathcal{V}}

\global\long\def\shW{\mathcal{W}}

\global\long\def\shX{\mathcal{X}}

\global\long\def\shY{\mathcal{Y}}

\global\long\def\shZ{\mathcal{Z}}

\global\long\def\st{\ | \ }

\global\long\def\stA{\mathcal{A}}

\global\long\def\stB{\mathcal{B}}

\global\long\def\stC{\mathcal{C}}

\global\long\def\stD{\mathcal{D}}

\global\long\def\stE{\mathcal{E}}

\global\long\def\stF{\mathcal{F}}

\global\long\def\stG{\mathcal{G}}

\global\long\def\stH{\mathcal{H}}

\global\long\def\stI{\mathcal{I}}

\global\long\def\stJ{\mathcal{J}}

\global\long\def\stK{\mathcal{K}}

\global\long\def\stL{\mathcal{L}}

\global\long\def\stM{\mathcal{M}}

\global\long\def\stN{\mathcal{N}}

\global\long\def\stO{\mathcal{O}}

\global\long\def\stP{\mathcal{P}}

\global\long\def\stQ{\mathcal{Q}}

\global\long\def\stR{\mathcal{R}}

\global\long\def\stS{\mathcal{S}}

\global\long\def\stT{\mathcal{T}}

\global\long\def\stU{\mathcal{U}}

\global\long\def\stV{\mathcal{V}}

\global\long\def\stW{\mathcal{W}}

\global\long\def\stX{\mathcal{X}}

\global\long\def\stY{\mathcal{Y}}

\global\long\def\stZ{\mathcal{Z}}

\global\long\def\then{\ \Longrightarrow\ }

\global\long\def\L{\textup{L}}

\global\long\def\l{\textup{l}}

\newcommand{\lp}{{\textup{(}}}
\newcommand{\rp}{{\textup{)}}}
\newcommand{\lb}{{\textup{[}}}
\newcommand{\rb}{{\textup{]}}}

\newcommand{\B}{{\mathbb B}}
\newcommand{\D}{{\mathbb D}}
\newcommand{\G}{{\mathbb G}}
\renewcommand{\H}{{\mathbb H}}
\newcommand{\I}{{\mathbb I}}
\newcommand{\J}{{\mathbb J}}
\newcommand{\M}{{\mathbb M}}
\renewcommand{\P}{{\mathbb P}}
\newcommand{\R}{{\mathbb R}}
\newcommand{\T}{{\mathbb T}}
\newcommand{\U}{{\mathbb U}}
\newcommand{\V}{{\mathbb V}}
\newcommand{\W}{{\mathbb W}}
\newcommand{\X}{{\mathbb X}}
\newcommand{\Y}{{\mathbb Y}}

\newcommand{\sA}{{\mathcal A}}
\newcommand{\sB}{{\mathcal B}}
\newcommand{\sC}{{\mathcal C}}
\newcommand{\sD}{{\mathcal D}}
\newcommand{\sE}{{\mathcal E}}
\newcommand{\sF}{{\mathcal F}}
\newcommand{\sG}{{\mathcal G}}
\newcommand{\sH}{{\mathcal H}}
\newcommand{\sI}{{\mathcal I}}
\newcommand{\sJ}{{\mathcal J}}
\newcommand{\sK}{{\mathcal K}}
\newcommand{\sL}{{\mathcal L}}
\newcommand{\sM}{{\mathcal M}}
\newcommand{\sN}{{\mathcal N}}
\newcommand{\sO}{{\mathcal O}}
\newcommand{\sP}{{\mathcal P}}
\newcommand{\sQ}{{\mathcal Q}}
\newcommand{\sR}{{\mathcal R}}
\newcommand{\sS}{{\mathcal S}}
\newcommand{\sT}{{\mathcal T}}
\newcommand{\sU}{{\mathcal U}}
\newcommand{\sV}{{\mathcal V}}
\newcommand{\sW}{{\mathcal W}}
\newcommand{\sX}{{\mathcal X}}
\newcommand{\sY}{{\mathcal Y}}
\newcommand{\sZ}{{\mathcal Z}}


\newcommand{\Aff}{{\rm Aff}}
\newcommand{\Aut}{{\rm Aut}}
\newcommand{\an}{{\rm an}}
\newcommand{\alg}{{\rm alg}}
\newcommand{\Bd}{{\rm Band}}
\newcommand{\Cats}{{\rm Cats}}
\newcommand{\ch}{\textup{Ch}}
\newcommand{\Char}{{\rm char}}
\newcommand{\codim}{{\rm codim}}
\newcommand{\cont}{{\rm cont}}
\newcommand{\Cov}{\textup{Cov}}
\newcommand{\Crys}{{\rm Crys}}
\newcommand{\cts}{\textup{cts}}
\newcommand{\Div}{{\rm Div}}
\newcommand{\Dmod}{{\rm Dmod}}
\newcommand{\ed}{{\rm ed}}
\newcommand{\Ess}{{\rm EFin}}
\renewcommand{\et}{\textup{\'et}}
\newcommand{\ev}{\textup{ev}}
\newcommand{\Fdiv}{{\rm Fdiv}}
\newcommand{\FEt}{{\textup{FÉt}}}
\newcommand{\Fib}{{\rm Fib}}
\newcommand{\FSets}{{\rm FSets}}
\newcommand{\Frac}{{\rm Frac}}
\newcommand{\FtAff}{{\rm FtAff}}
\newcommand{\Gal}{{\rm Gal}}
\newcommand{\height}{\textup{ht}}
\newcommand{\HOM}{{{\shH}\rm om}}
\newcommand{\Hom}{{\rm Hom}}
\newcommand{\iinf}{\textup{inf}}
\newcommand{\im}{{\rm Im}}
\newcommand{\Ker}{{\rm Ker}}
\newcommand{\LL}{\textup{L}}
\newcommand{\Loc}{{\rm Loc}}
\newcommand{\Max}{{\rm Max \ }}
\newcommand{\MIC}{\mbox{MIC}}
\newcommand{\Min}{{\rm Min \ }}
\newcommand{\NN}{\textup{N}}
\newcommand{\Mod}{\text{\sf Mod}}
\newcommand{\Noohi}{\textup{Noohi}}
\newcommand{\perf}{\textup{perf}}
\newcommand{\pet}{{\textup{proét}}}
\newcommand{\Pic}{{\rm Pic}}
\newcommand{\Rep}{\text{\sf Rep}}
\newcommand{\Res}{{\rm Res}}
\newcommand{\rank}{{\rm rank}}
\newcommand{\red}{{\rm red}}
\newcommand{\Spf}{\textup{Spf}}
\newcommand{\spe}{\textup{sp}}
\newcommand{\str}{\textup{str}}
\newcommand{\strat}{{\rm Str}}
\newcommand{\sym}{\text{Sym}}
\newcommand{\tp}{{\rm top}}
\newcommand{\Tr}{{\rm Tr}}
\newcommand{\trace}{{\rm Tr}}
\newcommand{\vect}{\text{\sf vect}}
\newcommand{\Vect}{\text{\sf Vect}}




\setcounter{section}{0}
\maketitle

\begin{abstract}
    We investigate how the étale fundamental group controls local systems in characteristic $p$, namely $F$-divided sheaves. In analogy with Grothendieck--Malcev’s results for discrete groups, we show that if a morphism $f \colon Y \to X$ of smooth projective varieties over $k=\bar{k}$ induces a surjection on the étale fundamental groups, then the pullback functor $\Fdiv(X)\to \Fdiv(Y)$ is fully faithful. If $f$ is surjective and the induced map is an isomorphism, then the functor is an equivalence. These results extend the theorem of Esnault--Mehta \cite[Theorem 1.1]{EM10} on the triviality of $F$-divided sheaves over simply connected varieties.
\end{abstract}


\section*{Introduction}

Let $X$ be a connected schemes of finite type over $\C$. By the Riemann
existence theorem, the étale fundamental group $\pi_1^\et(X,x)$ is the profinite
completion of the topological fundamental group $\pi_1^\tp(X^\an,x^\an)$.
Now, let $f\colon Y\to X$ be a map of connected schemes of finite
type over $\C$. If $f$ induces an isomorphism  $\pi_1^\tp(Y^\an,y^\an)\to
\pi_1^\tp(X^\an,x^\an)$, then it also induces an isomorphism on the étale fundamental groups
$\pi_1^\et(Y,y)\to \pi_1^\et(X,x)$. In fact, a stronger result holds: Let
$\Rep_{\C}(G)$ denote the category of finite-dimensional complex
representations of a group $G$. If the natural functor 
\begin{equation}\label{Ziel in complex case}
    \Rep_{\C}(\pi_1^\tp(X^\an,x^\an))\to\Rep_{\C}(\pi_1^\tp(Y^\an,y^\an)) \tag{$\star$}
\end{equation}is  fully faithful (resp. an equivalence), then the induced map
$\pi_1^\et(Y,y)\to \pi_1^\et(X,x)$ is surjective (resp. an isomorphism).
The converse is also true, though more subtle.

\begin{cor}[Grothendieck, Malcev] Let $f\colon X\to Y$ be a map of connected schemes of finite type
    over $\C$. 
    \begin{enumerate}[label={\rm (\arabic*)}]
        \item If $\pi_1^\et(X)\to\pi_1^\et(Y)$ is surjective, then \eqref{Ziel
    in complex case} is fully faithful.
        \item If $\pi_1^\et(X)\to\pi_1^\et(Y)$ is an isomorphism, then \eqref{Ziel
    in complex case} is an equivalence.
    \end{enumerate}
\end{cor}

This result is a direct consequence of a more general result due to Grothendieck
\cite{Gro70} and Malcev \cite{Mal40}:

\begin{thm}[Grothendieck, Malcev] Let $f\colon \Gamma_1\to\Gamma_2$ be a group homomorphism between two
    discrete
    groups. For a commutative ring $A$ and a discrete group $G$, denote by
    $\Rep_A(G)$ the category of finitely presented $A$-modules equipped with a
    $G$-action.
    \begin{enumerate}[label={\rm (\arabic*)}]
        \item If the projective completion $\hat{f}\colon\hat{\Gamma}_1\to\hat{\Gamma}_2$
            is surjective and $\Gamma_2$ is finitely generated, then the
            functor  
            \(\Rep_{A} (\Gamma_1)\to\Rep_{A} (\Gamma_2)\) is fully faithful for
            every $A$.
        \item If  $\hat{f}\colon\hat{\Gamma}_1\to\hat{\Gamma}_2$
            is an isomorphism and both $\Gamma_1,\Gamma_2$ are finitely
            generated, then the
            functor  
            \(\Rep_{A} (\Gamma_1)\to\Rep_{A} (\Gamma_2)\) is an equivalence for
            every $A$.
    \end{enumerate}
\end{thm}

In this paper, we primarily study a morphism $f\colon Y\to X$ of smooth, geometrically connected projective varieties over
    a perfect field $k$ of characteristic $p>0$. 
    Fix a $k$-rational point $\xi'$ of $Y$ and let $\xi\coloneqq f(\xi')$ be the
    corresponding rational point of $X$. We use $\pi_1^{\NN,\et}(Y,\xi')$ to
    denote the maximal étale quotient of the Nori fundamental group scheme.
    Recall that if $k$ is separably closed, then $\pi_1^{\NN,\et}(Y,\xi')$ is
    nothing but the étale fundamental group $\pi_1^\et(Y,\xi')$, and
    $\pi_1^{\NN,\et}$ satisfies base change for algebraic separable field extensions.
    \begin{thmI}[cf.~Theorem \ref{full faithfulness the general case}]\label{ThmI} If the induced map $\pi(f)\colon\pi_1^{\NN,\et}(Y,\xi')\to
        \pi_1^{\NN,\et}(X,\xi)$ is surjective, then the pullback functor 
        \[
            \Fdiv(X)\longrightarrow \Fdiv(Y)
        \]
        is fully faithful.
\end{thmI}

If the geometric fundamental group
$\pi^\et_1(X_{\bar{k}},\xi)=0$ and $\Spec(l)\to X$ is a closed point, then
applying 
Theorem \ref{ThmI} to the map $\Spec(l)\to X_l$ implies that $\Fdiv(X_l)$ is
trivial. By base change (cf.~\cite[Appendix, Prop.~B3]{TZ2}),
$\Fdiv(X_l)=\Fdiv(X)\otimes_kl$, so $\Fdiv(X)$ is trivial. This
recovers the main result of \cite[Theorem 1.1]{EM10}.


\begin{thmII}[cf.~Theorem \ref{separable for general bundles over general
    fields final version}]\label{ThmII} Suppose that $f$ is surjective. If $f$
    induces an isomorphism $\pi(f)\colon\pi_1^{\NN,\et}(Y,\xi')\to
        \pi_1^{\NN,\et}(X,\xi)$, then the induced map
        \[\pi^{\Fdiv}_1(Y,\xi')\longrightarrow\pi^{\Fdiv}_1(X,\xi)\] is an
        isomorphism as well.
\end{thmII}

Suppose that the geometric fundamental group
$\pi^\et_1(Y_{\bar{k}},\xi')=0$, and let $Y\to X=\Spec(k)$ be the structure map.
Then by
Theorem \ref{ThmII}, the category $\Fdiv(Y)$ is trivial. This result again recovers \cite[Theorem 1.1]{EM10}.


\textbf{Methodology.}
The proofs of both Theorem I and Theorem II build upon the framework developed
in \cite{EM10} and \cite{Sun19}, leveraging the theory of representation spaces.
\begin{itemize}
    \item 
        For Theorem \ref{ThmI}, we reinterpret morphisms between $F$-divided sheaves as sequences of Hom-sheaves on representation spaces. This allows us to reduce the problem to studying geometric properties of these moduli spaces.
    \item For Theorem \ref{ThmII}, however, new challenges arise: Essential
        surjectivity does not translate cleanly into a moduli-theoretic
        formulation, primarily due to complications caused by purely
        inseparable morphisms. To overcome this, we factor $f$ as a composition of a purely inseparable morphism followed by a separable morphism and treat each case separately. This approach forces us to work with singular varieties, which we handle using ad-hoc techniques tailored to the specific geometric and arithmetic properties of the maps involved.
\end{itemize}
The proof of Theorem II is complicated by the necessity of dealing with
singular varieties, and it remains unknown to us whether a systematic approach to handling such singularities is available in this context.

\textbf{Existing literature.} 
Our work builds upon the foundational results of Esnault--Mehta~\cite{EM10}. 
In Sun~\cite{Sun19} and Esnault--Srinivas~\cite{Esnault-Srinivas2016}, the authors proved that if a morphism \( f \colon Y \to X \) induces the zero map on étale fundamental groups, then it also induces the zero map on the \(F\)-divided fundamental group. 

In~\cite{Esnault-Srinivas2019}, Esnault--Srinivas showed that for a normal projective variety \( X \) with regular locus \( U \), the condition \(\pi_1^{\text{ét}}(U) = 0\) implies \(\pi_1^{\Fdiv}(U) = 0\). 

Theorem~\ref{ThmII} was proved in the rank \(1\) case by
Kindler~\cite{Kindler2013}, and more recently for generically étale finite surjective morphisms by Biswas--Kumar--Parameswaran in~\cite{BKP25}. 
Our results extend these works to a broader geometric setting.

\section*{\centering Notation and Conventions}\refstepcounter{NC}\label{s:notation}
\begin{enumerate}[label=(\arabic*)]
    \item A morphism $Y \to X$ of integral schemes is called \emph{separable}
        if the generic point of $Y$ maps to the generic point of $X$, and the
        induced extension of function fields $K(X)\subseteq K(Y)$ is separable
        (possibly infinite), in the sense of \cite[\href{https://stacks.math.columbia.edu/tag/030I}{Tag 030I}]{stacks-project} (see also \cite[\href{https://stacks.math.columbia.edu/tag/030W}{Tag 030W}]{stacks-project}).

    \item Let $X$ be a projective scheme. We will usually fix, implicitly, a polarization on $X$, i.e. a very ample line bundle. Stability notions such as $\mu$-stability (slope stability) and $p$-stability (Gieseker stability) will always be taken with respect to this fixed polarization.

    \item We call a homomorphism of affine group schemes $G\to G'$
        \emph{surjective} if it is faithfuly flat. This is equivalent to saying
        that it is surjective as fpqc-sheaves.

    \item In this paper, we'll mostly fix a positive integer $r \in \N^+$ and consider
    the representation space $R_{X,r} \coloneqq R(X, \xi, rp_{\sO_X})$, where $p_{\sO_X}$ denotes the
    Hilbert polynomial of $\sO_X$. When the rank $r$ is clear or the
    same from the context, we'll simply write $R_X$ for $R_{X,r}$.
\end{enumerate}

\section{Preliminaries}
\subsection{The Nori fundamental group and gerbe}

Let $k$ be a field, and let $X$ be a geometrically reduced and  
geometrically connected scheme of finite type over $k$.  
By \cite[Prop.~5.5, Thm.~5.7]{BV15}, $X$ admits a
\emph{Nori fundamental gerbe} $\Pi_{X/k}^{\NN}$ together with a $1$-morphism
\(
u\colon X \longrightarrow \Pi_{X/k}^{\NN}
\)
which classifies finite stacks \cite[Def.~3.1]{TZ1}.  
That is, for any finite stack $\Gamma$ over $k$, the pullback along $u$ induces an equivalence 
\begin{equation}\label{Nori gerbe}
    \Hom_k(\Pi_{X/k}^{\NN},\Gamma)\xrightarrow{\;\;\cong\;\;}\Hom_k(X,\Gamma).
\end{equation}

If $X$ admits a $k$-rational point $\xi$, then $\Pi_{X/k}^{\NN}$ is neutralized by the point $u(\xi)$.  
In this case, one obtains a profinite $k$-group scheme 
\(
\pi_1^{\NN}(X,\xi),
\)
the \emph{Nori fundamental group} of $X$ at $\xi$, associated with the neutral gerbe $(\Pi_{X/k}^{\NN},u(\xi))$.  
Taking $\Gamma$ in \eqref{Nori gerbe} to be the classifying stack $\sB_k G$ for a finite $k$-group scheme $G$, we obtain:
\begin{equation}\label{Nori fundamental group}
    \Hom_{k\text{-grp.sch}}(\pi_1^{\NN}(X,\xi),G)\xrightarrow{\;\;\cong\;\;}
    \{\text{pointed $G$-torsors over $(X,\xi)$}\}/\cong,
\end{equation}
where the right-hand side denotes the set of isomorphism classes of pairs $(P,p)$ with $P$ a $G$-torsor over $X$ and $p\in P(k)$ lying above $\xi$. 

If $X$ is moreover \emph{proper}, then by \cite[Thm.~7.9]{BV15} there is a natural equivalence of neutral Tannakian categories
\begin{equation}\label{essentially finite VB}
    \varpi \colon \Rep_k(\pi_1^{\NN}(X,\xi))\xrightarrow{\;\;\simeq\;\;} \EFin(X),
\end{equation}
where $\Rep_k(\pi_1^{\NN}(X,\xi))$ denotes the category of finite-dimensional representations of the Nori fundamental group, and $\EFin(X)$ the category of \emph{essentially finite vector bundles} on $X$ (cf.~\cite[Def.~7.7]{BV15}).  

The category $\Rep_k(\pi_1^{\NN}(X,\xi))$ is equipped with the forgetful fiber functor $F$, while $\EFin(X)$ is equipped with the fiber functor $\xi^*$ given by pullback to $\xi$.  
The equivalence $\varpi$ is compatible with fiber functors: there is an isomorphism
\(
\beta\colon \xi^*\circ\varpi \;\;\simeq\;\; F.
\)

The construction of the equivalence \eqref{essentially finite VB} proceeds as follows.  
Given a representation $\rho\colon \pi_1^{\NN}(X,\xi)\to \GL_V$, its \emph{monodromy group}
\(
G \coloneqq \operatorname{im}(\rho)
\)
is a finite subgroup scheme of $\GL_V$.  
The quotient map $\pi_1^{\NN}(X,\xi)\to G$ corresponds, via \eqref{Nori
fundamental group}, to a pointed $G$-torsor $(P,p)$, which extends naturally to
a pointed $\GL_V$-torsor.  
The associated vector bundle $V_\rho$ is defined to be $\varpi(\rho)$.

\subsection{The Nori-étale fundamental group}
The
\emph{Nori-étale fundamental group}
\resizebox{.11\textwidth}{!}{$\pi_1^{\NN,\et}(X,\xi)$} is the pro-étale completion of the group
scheme $\pi_1^{\NN}(X,\xi)$. Restricting 
 \eqref{Nori fundamental group} to a finite étale $k$-group scheme $G$, we have
 a bijection:
\begin{equation}\label{Nori-étale fundamental group}
\Hom_{k\text{-grp.sch}}(\pi_1^{\NN,\et}(X,\xi),G)\xrightarrow{\quad\cong\quad}\{\text{Pointed
    $G$-torsors over $(X,\xi)$}\}/\cong
\end{equation}
The equation
\eqref{Nori-étale fundamental group}  is reminiscent of the universal
property of the usual étale fundamental group $\pi_1^{\et}(X,\xi)$. Indeed, if
$k$ is separably closed,  then $\pi_1^{\et}(X,\xi)$ is nothing but the
$k$-points of 
$\pi_1^{\NN,\et}(X,\xi)$ equipped with the Zariski topology. Conversely, one
obtains $\pi_1^{\NN,\et}(X,\xi)$ out of $\pi_1^{\et}(X,\xi)$ by writing it as
a pro-system of finite groups and taking the limit in the category of constant
$k$-group schemes.

If $X$ is moreover proper, then equivalence \eqref{essentially finite VB} restricts to
    a fully faithful functor
    \begin{equation}\label{essentially finite VB, étale}
        \varpi\colon 
        \Rep_k(\pi_1^{\NN,\et}(X,\xi))\xrightarrow{\quad \subseteq\quad}
        \textrm{EFin}(X)
    \end{equation}
\begin{defn} Let $\EFinet(X)$ denote the essential image of the functor \eqref{essentially finite VB, étale}.
    A vector bundle $V\in\EFinet(X)$ is called an \emph{essential finite vector
    bundle with étale monodromy}. In this
    case, we also say $V$ \emph{comes from an étale representation}. 
\end{defn}

\begin{rmk}\label{etale trivializable}
If $V \in \textrm{EFin}(X)$ has étale monodromy, then the Tannakian subcategory
generated by $V$  corresponds to a finite étale group
scheme $G$; then $V$ is trivialized by a $G$-torsor, which is a finite étale
cover of $X$. Conversely, if $V$ is a vector bundle trivialized by a finite
étale cover, then it is trivialized by some étale Galois cover, so 
$V$ is essentially finite with étale monodromy. Therefore, a vector bundle 
$V$ is in $\EFinet(X)$ if and only if it is \emph{étale trivializable}, 
\emph{i.e.}, it can be trivialized by some finite étale surjective cover of $X$.\end{rmk}

Given any $\rho\in
\Hom_{k-\text{grp.sch}}(\pi_1^{\NN,\et}(X,\xi),\GL_n)$, $\rho$ corresponds, via
\eqref{essentially finite VB, étale}, to  an essentially finite vector bundle
with étale monodromy $V_{\rho}\coloneqq \varpi(\rho)$. Since $\rho$ is a representation on $k^{\oplus
n}$, the isomorphism of fiber functors provides an isomorphism
$\beta\colon \xi^*V_\rho\cong k^{\oplus n}$.
We call the pair $(V_\rho,\beta)$ a \textit{framed bundle}
 on $(X,\xi)$. In this way, we 
obtain a map of sets:
    \begin{equation}\label{representations to bundles}
        \Hom_{k\textup{-grp.sch}}(\pi_1^{\NN,\et}(X,\xi),\GL_n)\xrightarrow{\quad\psi_X\quad}\{\textup{Framed
    bundles $(V,\beta)$ on $(X,\xi)$}\}/\cong
    \end{equation}

\begin{lem}\label{the map psi} If $X$ is a geometrically reduced and geometrically
    connected proper scheme over $k$ equipped with a rational point $\xi\in
    X(k)$, then for any $n\in\N^+$
    \begin{enumerate}[label={\rm(\Alph*)}]
        \item the natural map $\psi_X$ in \eqref{representations to bundles} is
            injective;
        \item its image consists of isomorphism classes of pairs $(E,\beta)$, where
            $E$ is a rank $n$ étale trivializable bundle.
    \end{enumerate}
\end{lem}
\begin{proof} By construction, $\psi(\rho)$ is a pair $(E,\beta)$, where $E$ is
    essentially finite with étale monodromy. Conversely, given a framed bundle
    $(E,\beta)$ with $E$ essentially
    finite having étale monodromy, there is a representation
    $\rho\colon\pi_1^{\NN,\et}(X,\xi)\to\GL_V$ corresponding to $E$ via
    \eqref{essentially finite VB, étale}. Moreover, thanks to the compatibility
    of $\varpi$ with 
    fiber functors, we have  $\xi^*E\cong V$. Now the frame $\beta$ provides an
    isomorphism of $k$-spaces $\gamma\colon V\cong k^{\oplus n}$. This
    isomorphism induces a representation $\rho'\colon
    \pi_1^{\NN,\et}(X,\xi)\to\GL_n$, which is isomorphic to $\rho$ in $\Rep_k(\pi_1^{\NN,\et}(X,\xi))$, such that
    $\psi_X(\rho')\cong(E,\beta)$. This shows (B). 

    To show (A), suppose
    $\rho_1,\rho_2\in\Hom_{k-\textup{grp.sch}}(\pi_1^{\NN,\et}(X,\xi),\GL_n)$
    such that $\psi_X(\rho_1)=\psi_X(\rho_2)$. Set $\psi_X(\rho_i)\coloneqq
    (V_{\rho_i},\beta_i)$, $i=1,2$. Then we have an isomorphism $\lambda\colon
    V_{\rho_1}\cong V_{\rho_2}$ such that
    $\xi^*\lambda\colon\xi^*V_{\rho_1}\cong \xi^*V_{\rho_2}$ is compatible with
    $\beta_i\colon \xi^*V_{\rho_i}\cong k^{\oplus n}$. By \eqref{essentially
    finite VB, étale}, $\lambda$ corresponds uniquely to an isomorphism
    $\rho_1\cong \rho_2$ viewing as representations on $k^{\oplus n}$. Thus
    $\rho_1$ and $\rho_2$ differ at most by an inner automorphism of $\GL_n$.
    However, the assumption that $\beta_2\circ \xi^*\lambda=\beta_1$ implies
    that the choice of the basis are the same, so the inner automorphism must
    be
    trivial, \emph{i.e.} $\rho_1=\rho_2$ as homomorphisms.
\end{proof}

\begin{defn} We will call a framed bundle $(E,\beta)$ on $X$ \emph{étale
    trivializable} if $E$ lies in the image of \eqref{representations to
    bundles}.
\end{defn}

An important property of the Nori-étale fundamental group is that it satisfies
base change by algebraic separable field extensions.
\begin{prop}\label{Base change for Nori} Let $X$ be a geometrically reduced and geometrically
    connected proper scheme over $k$ equipped with a rational point $\xi\in
    X(k)$, and let $k\subseteq l$ be a \textup{(possibly infinite)} algebraic separable field
    extension. The natural map
    $\pi_1^{\NN,\et}(X\times_kl,\xi\times_kl)\to
    \pi_1^{\NN,\et}(X,\xi)\times_kl$ is an isomorphism.
\end{prop}
\begin{proof} It is proved in \cite[Chapter II, Prop. 5]{Nori82} and
\cite[§6]{BV15} that the Nori fundamental group satisfies base
change for algebraic separable extensions. By replacing ``finite'' with 
``finite étale'' and $\pi_1^\NN$ with $\pi_1^{\NN,\et}$, the same proof applies verbatim to our case.
\end{proof}

\begin{lem} \label{separable fiber unique} Assume that $X$ is a normal scheme proper over a
    field $k$, and that $f\colon Y\to X$  is a proper surjective separable map
    between two connected geometrically reduced $k$-schemes. Suppose that  $\pi(f)\colon \pi_1^{\NN,\et}(Y,\xi')\to\pi_1^{\NN,\et}(X,\xi)$ is
    surjective, where $\xi\in X(k)$, $\xi'\in Y(k)$ are rational points such
    that
    $f(\xi')=\xi$. Let $V$ be a vector bundle on $X$.
    If $f^*V$  is étale trivializable, then so is $V$.
\end{lem}
\begin{proof} Using the Stein factorization
(cf.~\cite[\href{https://stacks.math.columbia.edu/tag/03GX}{03GX}]{stacks-project}
and
\cite[Exposé X, Prop.~1.2]{SGA1}), we factor $f$ as 
$Y \xrightarrow{h} Z \xrightarrow{g} X$,
where $\sO_Z \xrightarrow{\cong} h_*\sO_Y$ and $g$ is finite and generically
étale. 
Consequently, the pullback functor $h^*$ from the category of essentially finite
vector bundles on $Z$ to that on $Y$ is fully faithful. Since $f^*V = h^*(g^*V)$ 
is étale trivializable, it follows that $g^*V$ is also étale trivializable. 
Thus, without loss of generality, we may assume $f = g$, i.e., $f$ is finite and 
generically étale.
Since $f^*V$ is étale trivializable, there exists a finite étale cover 
$g\colon X' \to X$ that trivializes $f^*V$. Hence, $V$ is trivialized by the 
generically étale morphism $f \circ g$. The result now follows from
\cite[Thm. 1 (ii)]{BDS12}, \cite[Thm. 1.1]{AM11}, see also 
\cite[Thm. II, 2]{TZ3} for the case where $k$ is not algebraically closed.
\end{proof}

\subsection{The representation space} 
The \emph{representation space} $R_X$ is a fine moduli space classifying
$p$-semistable vector bundles on $X$ with fixed Hilbert polynomial together
with a framing at a basepoint. Unlike coarse moduli spaces, which parametrize
$S$-equivalence classes, 
this space distinguishes between non-isomorphic bundles through the additional framing data.
This construction provides essential geometric tools for studying the relationship between
fundamental groups and bundle theory via Tannakian methods.

Let $S$ be an affine variety over $\F_q$, and let $X_S \to S$ be a flat morphism with geometrically integral fibers. Fix a polynomial $P$ of degree $d \coloneqq \dim(X_S/S)$ and a relative ample line bundle $\sO_{X_S}(1)$ on $X_S$. Suppose that $\xi_S\colon S \to X_S$ is an $S$-section. 

\begin{defn} \label{moduli condition}
    \begin{enumerate}
        \item A coherent sheaf $\sF$ on $X_S$ is called \textit{$p$-semistable with Hilbert polynomial $P$} if it is flat over $S$ and $\sF_s$ is $p$-semistable with Hilbert polynomial $P$ on each geometric fiber $X_s$ of $X_S \to S$.
        \item We say that $\sF$ satisfies condition $\mathrm{LF}(\xi_S)$ if $\gr(F_s)$ is locally free at $\xi_S(s)$ for any geometric point $s$ of $S$.
    \end{enumerate}
\end{defn}

\begin{thm}
    There exists an $S$-scheme $R(\sO_{X_S}, \xi_S, P)$ that represents the functor associating to any $T \to S$ the set of framed sheaves $(\sE, \beta)$, where:
    \begin{itemize}
        \item $\sE$ is a $p$-semistable torsion-free sheaf of Hilbert polynomial $P$ on $X_T/T$ satisfying condition $\mathrm{LF}(\xi_T)$,
        \item $\beta\colon \xi_T^*(\sE) \cong \sO_T^{\oplus r}$ is a frame.
    \end{itemize}
    The fine moduli space $R(\sO_{X_S}, \xi_S, P)$ is called \emph{the representation space}.
\end{thm}
\begin{proof} 
    See \cite[Theorem 2.3]{Sun19}.
\end{proof}

\subsection{$F$-divided sheaves}
\emph{$F$-divided sheaves} are a characteristic $p$ analogue of local systems. An $F$-divided sheaf is a sequence $(E_i, \sigma_i)_{i \in \mathbb{N}}$ of vector bundles $E_i$ together with isomorphisms $\sigma_i: F^* E_{i+1} \xrightarrow{\sim} E_i$, where $F$ is the Frobenius morphism. This structure captures the notion of a ``sheaf with flat connection'' in positive characteristic, providing a Tannakian category equivalent to representations of the topological fundamental group for smooth projective complex varieties.

\begin{defn} 
    Let $X$ be a scheme over a field $k$ of characteristic $p > 0$. Denote by $\Fdiv(X)$ the category whose
    \begin{itemize}
        \item \textbf{objects are:} $(E_i, \sigma_i)_{i \in \mathbb{N}}$, where $E_i$ is a finitely presented sheaf on $X$ and $\sigma_i$ is an isomorphism $F_X^* E_i \xrightarrow{\cong} E_{i-1}$;
        \item \textbf{morphisms are:} $(f_i)_{i \in \mathbb{N}} \colon (E_i, \sigma_i)_{i \in \mathbb{N}} \to (E_i', \sigma_i')_{i \in \mathbb{N}}$, where each $f_i$ is a map of $\sO_X$-modules and $\sigma_i' \circ F_X^* f_i = f_{i-1} \circ \sigma_i$.
    \end{itemize}
\end{defn}

\begin{rmk} 
    Note that if $X$ is a geometrically connected scheme of finite type over a perfect field $k$ of characteristic $p > 0$, then each $E_i$ is a vector bundle. In particular, $\Fdiv(X)$ is a Tannakian category over $k$ (cf.~\cite[\S 2.2]{dS2007} and \cite[Theorem I (3)]{TZ1}). Thus we have the following definition.
\end{rmk}

\begin{defn}
    Let $X$ be a connected scheme of finite type over a perfect field $k$ of
    $\Char(k)>0$, and
    let $x \in X(k)$ be a $k$-rational point. We denote by $\pi_1^\Fdiv(X, x)$
    the affine group scheme corresponding to the neutral Tannakian category
    $\Fdiv(X)$ together with the neutral fiber functor $\omega_x$ sending
    $(E_i,\sigma_i)_{i\in\N}\mapsto x^*E_0$. 
\end{defn}

\begin{thm}\label{etale trivializable bundles are stratified bundles} 
    Suppose that $X$ is geometrically connected proper over a perfect field $k$ of
    $\Char(k)>0$. There is a fully faithful functor $\EFinet(X) \to \Fdiv(X)$ whose essential image consists precisely of the essentially finite objects \textup{(cf.~\cite[Def. 7.7]{BV15})} of $\Fdiv(X)$.
\end{thm}
\begin{proof} 
    This is a consequence of \cite[Thm.~6.23, Thm.~5.8]{TZ1}.
\end{proof}

An important result that we are going to use repeatedly is the following:

\begin{thm}\label{EM10 main lemma}
Let $X$ be a smooth connected projective scheme over a perfect field of characteristic $p > 0$.  
For any $E \coloneqq (E_i, \sigma_i)_{i \in \mathbb{N}} \in \Fdiv(X)$, the following hold:
\begin{enumerate}[label={\rm(\arabic*)}]
    \item The Chern classes of each $E_n$ $(n \geq 0)$ vanish.
    \item There exists $n_0 \in \mathbb{N}$ such that the \emph{$n_0$-shifted object}
    \(
        E(n_0) \coloneqq (E_{i+n_0}, \sigma_{i+n_0})_{i \in \mathbb{N}}
    \)
    is a successive extension of stratified bundles $U = (U_i, \delta_i)_{i \in \mathbb{N}}$, with each $U_i$ $\mu$-stable.
\item[\textup{(2')}] Suppose that $X$ is \emph{normal} with regular locus $U \subseteq
    X$. If $E(n)$ is not stable for any $n\in\N$, then there exists
    $n_0\in\N$ and a sequence of stable subsheaves
    \(
        (F_i \subseteq E_i(n_0))_{i \in \mathbb{N}}
    \)
    which form a proper $F$-divided subsheaf restricting to $U$.
    \item Each $E_n$ is $p$-semistable with reduced Hilbert polynomial $p_{\sO_X}$.
\end{enumerate}
\end{thm}

\begin{proof}
Statements (1), (2), and (3) follow from \cite[Lem.~2.1, Cor.~2.2, Prop.~2.3]{EM10} together with the 
\href{https://page.mi.fu-berlin.de/esnault/preprints/helene/95-erratum-prop3.2.pdf}{Erratum}.  
For (2'), one adapts the argument in loc.~cit. by noting that the Frobenius pullback is only flat on $U$.  
Instead of taking the Frobenius pullback of the socle directly, one considers
its image.  
Since this image agrees with the original sheaf at all points of codimension
$\leq 1$, the $\mu$-slope and $\mu$-stability are preserved. In this way, one
finds $n_0\in\N$ and  a sequence of stable subsheaves
    \(
        (F_i \subseteq E_i(n_0))_{i \in \mathbb{N}}
    \)
    which form an $F$-divided subsheaf restricting to $U$. If
    $(F_i|_U)_{i\in\N}=E(n_0)|_U$, then $E_i$ are stable for $i\geq n_0$
    contradicting to the assumption. 
\end{proof}

\begin{cor}\label{make F-divided sheaves points of the representation space} 
    Let $X$ be a smooth connected projective scheme defined over a perfect
    field $k$ of characteristic $p > 0$, and let $x$ be a $k$-point of $X$. Then for any $E \in \Fdiv(X)$ of rank $r$, there is $n_0 \in \mathbb{N}$ and a compatible choice of trivializations $\beta_i \colon x^* E_i \cong \sO_k^{\oplus r}$ such that $(E_i, \beta_i) \in R(X, x, rp_{\sO_X})(k)$ for $i \geq n_0$.
\end{cor}
\begin{proof} 
    Pick any $\beta_0 \colon x^* E_0 \cong \sO_k^{\oplus r}$; then $\beta_i$ ($i \geq 1$) are uniquely determined simply because $k$ is perfect.
\end{proof}

In view of Theorem \ref{etale trivializable bundles are stratified bundles}, Corollary \ref{EM10 main lemma} implies that every étale trivializable bundle $E$ on $X$, when equipped with a frame $\beta \colon x^* E \cong \sO_k^{\oplus r}$, can be viewed as a rational point of the representation space $R(X, x, rp_{\sO_X})$. More generally, we have:

\begin{lem} 
    Let $X$ be a normal connected projective scheme over a perfect field $k$, and let $x$ be a smooth $k$-point of $X$. Then any étale trivializable sheaf $E$ together with a choice of a frame $\beta$ can be viewed as a $k$-rational point of $R(X, x, rp_{\sO_X})$. 
\end{lem}
\begin{proof} 
    One has to check the two conditions in Definition \ref{moduli condition}. (1) is well-known (cf.~\cite[\S 1.2]{Langer12}) -- any essentially finite vector bundle of rank $r$ is semistable with reduced Hilbert polynomial $rp_{\sO_X}$. (2) follows from \cite[Theorem 4.8]{Langer24}.
\end{proof}

The following lemma is adapted from \cite[Theorem~3.7]{Sun19} and provides a way to approximate $F$-divided sheaves via étale trivializable sheaves in the moduli spaces.

Recall that for $r\in\N^+$, we denote $R_{X,r}\coloneqq R(X,\xi,rp_{\sO_X})$.
For an $n$-tuple $\vec{r}\coloneqq(r_1,\dots,r_n)$, set $
R_{X,\vec{r}}\coloneqq R_{X,r_1}\times\cdots\times R_{X,r_n}$.

\begin{lem}\label{Definition of the subscheme N} 
    Let $X$ be a connected projective scheme over an algebraically closed field $k$ equipped with a $k$-rational point $\xi$. Let $\Sigma \coloneqq \{(Q_i^{(1)}, \cdots, Q_i^{(n)})\}_{i \in \mathbb{N}}$ be a collection of $k$-points of the moduli space $R_{X,\vec{r}}$ such that $F_X^* Q_{i+1}^{(1)} = Q_i^{(1)}$, $\cdots$, $F_X^* Q_{i+1}^{(n)} = Q_i^{(n)}$. There is a closed integral subscheme $\sN \subseteq R_{X,\vec{r}}$ such that:
    \begin{enumerate}[label=\textup{(\Alph*)}]
        \item $\sN(k)$ contains $(Q_i^{(1)}, \cdots, Q_i^{(n)})$ for infinitely many $i \in \mathbb{N}$;
        \item $\Sigma \cap \sN(k)$ is dense in $\sN$;
        \item There is some $a \in \mathbb{N}^+$ such that the Frobenius pullback of the moduli space induces a dominant rational map $(F_X^*)^a \colon \sN \dashrightarrow \sN$;
        \item If $k = \bar{\F}_p$, then $\sN(k)$ contains a dense subset of points of the form $(Q_i^{(1)}, \cdots, Q_i^{(n)})$ with $Q_i^{(1)}, \cdots, Q_i^{(n)}$ being periodic.
    \end{enumerate}
\end{lem}
\begin{proof}
    Let's concentrate on $n = 2$ and $r_1=r_2$ -- other cases are completely analogous. Consider the subset 
    \[
    R_m \coloneqq \Set{(P_i, Q_i) | i \geq m} \subseteq R_X \times_k R_X.
    \]
    We have a descending chain 
    \[
    R_0 \supseteq R_1 \supseteq R_2 \supseteq \cdots \supseteq R_m \supseteq R_{m+1} \supseteq \cdots
    \]
    satisfying $F_X^*(R_{i+1}) = R_i$. Let $\sZ_i \coloneqq \overline{R_i} \subseteq R_X \times_k R_X$ be the Zariski closure of $R_i$ with reduced closed subscheme structure. We have a descending chain:
    \[
    \sZ_0 \supseteq \sZ_1 \supseteq \cdots \supseteq \sZ_m \supseteq \sZ_{m+1} \supseteq \cdots
    \]
    Since $R_X$ is Noetherian, there is $n_0 > 0$ such that $\sZ_{n_0} = \sZ_{n_0+1} = \cdots$. Let 
    \[
    \sZ \coloneqq \bigcap_{i=1}^{\infty} \sZ_i \subseteq R_X \times_k R_X.
    \]
    We have $R_{n_0} \subseteq \sZ$ is dense, and $F_X^* \colon \sZ \dashrightarrow \sZ$ is a dominant rational map \cite[Prop. 2.5]{Sun19}. Thus there is an irreducible component $\sN \subseteq \sZ$ such that:
    \begin{itemize}
        \item $\sN(k)$ contains $(P_i, Q_i)$ for infinitely many $i \in \mathbb{N}$;
        \item There is an integer $a > 0$ such that $(F_X^*)^a \colon \sZ \dashrightarrow \sZ$ induces a dominant rational map $(F_X^*)^a \colon \sN \dashrightarrow \sN$.
    \end{itemize}
    
    If $k = \bar{\F}_p$, then applying Lemma \cite[Lemma 3.6]{Sun19}, the subset $\Frob \subseteq \sN(k)$ consisting of points $(\sQ, \sQ') \in \sN(k)$ such that $(F_X^*)^m \sQ = \sQ$ and $(F_X^*)^m \sQ' = \sQ'$ for some $m \in \mathbb{N}^+$ is dense. This is (D).

    Let $\{Z_i\}_{i=1,\dots,n}$ be the irreducible components of $\sZ$ which are not $\sN$. Then $\sU = \sZ \setminus (\cup_{i=1,\dots,n} Z_i)$ is an open dense (resp. open) subset of $\sN$ (resp. $\sZ$). Since $\Sigma \cap \sZ(k)$ is dense in $\sZ$, $\Sigma \cap \sU(k)$ is dense in $\sU$. Therefore, $\Sigma \cap \sN(k)$ is dense in $\sN$. This is (B).
\end{proof}


The following lemma shows that $\mu$-stability descends along surjective maps.

\begin{lem}\label{pullback of stable is stable}
    Let $f \colon Y \to X$ be a surjective map of normal projective integral schemes of finite type over an algebraically closed field $k$, where $Y$ is smooth. Suppose $E \in \Fdiv(X)$ is an $F$-divided sheaf. If each $f^* E_i$ is $\mu$-stable \textup{(}for any chosen polarization\textup{)}, then $E_i$ is $\mu$-stable for $i \gg 0$ and for any polarization.
\end{lem}
\begin{proof}
    Suppose $E_i(n)$ was not $\mu$-stable for all $n \in\N$. Then by Theorem
    \ref{EM10 main lemma} (2'), there is a sequence of stable subsheaves
    $(F_i)_{i \in \mathbb{N}}$ of $(E_i, \sigma_i)$ which form a  proper
    $F$-divided subsheaf restricting to the regular locus $U$ of $X$. Then we obtain an $F$-divided subsheaf $(f^* F_i)|_{U'} \subseteq f^* E|_{U'}$, where $U' \coloneqq f^{-1}(U)$. Thanks to \cite[Lemma 2.5]{Kindler2015}, $(f^* F_i)|_{U'}$ extends to an $F$-divided subsheaf of $f^* E$. Considering that $f^* E_i$ is stable, $f^* E$ must be the aforementioned extension. This implies that $F_i|_U = E_i|_U$ for all $i \in \mathbb{N}$, a contradiction! 
\end{proof}


\subsection{Artin-Schreier theory for framed bundles}
The following result is often referred to as the Lange-Stuhler theorem. However, it is a consequence of the ``non-abelian Artin-Schreier theory'' (cf.~\cite[\S 3.2]{Laszlo01} and \cite[\S 1]{BD07}), due to Deligne, which is based on the Lang isogeny.

\begin{lem}\label{Lange-Stuhler}
    Let $X$ be a scheme over $k$, and let $F_X \colon X \to X$ be the absolute Frobenius morphism. If there exists a vector bundle $\mathcal{E}$ on $X$ and an integer $m > 0$ such that $(F_X^*)^m \mathcal{E} \cong \mathcal{E}$, then $\mathcal{E}$ is étale trivializable.
\end{lem} 

\begin{defn} 
    Let $(E, \beta)$ be a framed bundle on $X$. If $(E, \beta) \cong (F_X^*)^a (E, \beta) \coloneqq ((F_X^*)^a E, (F_k^*)^a \beta)$ for some $a \in \mathbb{N}^+$, then we say $(E, \beta)$ is \emph{periodic}.  
\end{defn}

As we will see, Lemma \ref{Lange-Stuhler} implies that periodic framed bundles are étale trivializable.

Now assume that $f \colon Y \to X$ is a morphism between geometrically connected and geometrically reduced projective varieties over a perfect field $k$ of characteristic $p > 0$.

\begin{cor}\label{cor3.2} 
    Suppose that the induced map $\pi(f) \colon \pi_1^\et(Y, \xi') \to \pi_1^\et(X, \xi)$ is an isomorphism. Let $(E, \beta)$ be an étale trivializable framed bundle on $Y$. 
    \begin{enumerate}[label={\rm(\Alph*)}]
        \item If $(E, \beta)$ is periodic, then there exists a unique homomorphism $\rho \colon \pi_1^\et(Y, \xi') \to \GL_n$ such that $\psi_Y(\rho) = (E, \beta)$;
        \item There exists a framed bundle $(F, \alpha)$ on $X$, where $F$ is étale trivializable (cf.~Remark \ref{etale trivializable}), such that $(f^* F, f^* \alpha) \cong (E, \beta)$;
        \item If $(F, \alpha)$ and $(F', \alpha')$ are two framed bundles on $X$ with $F, F'$ étale trivializable, then $(f^* F, f^* \alpha) \cong (f^* F', f^* \alpha')$ implies $(F, \alpha) \cong (F', \alpha')$.
    \end{enumerate}
\end{cor}

\begin{proof} 
    (A) By assumption, $E \cong (F_X^*)^a E$; thus by Lemma \ref{Lange-Stuhler}, the bundle $E$ is étale trivializable. In view of Lemma \ref{the map psi}, the pair $(E, \beta)$ lies in the image of $\psi_Y$. Therefore, there exists a unique homomorphism $\rho \colon \pi_1^\et(Y, \xi') \to \GL_n$ such that $\psi_Y(\rho) = (E, \beta)$.

    (B) Set $\rho_0 \coloneqq \rho \circ \pi(f)^{-1} \colon \pi_1^\et(X, \xi) \to \GL_n$, and define $(F, \alpha) \coloneqq \psi_X(\rho_0)$. Since $\rho_0 \circ \pi(f) = \rho$, we have $f^* \psi_X(\rho_0) = \psi_Y(\rho)$ by the naturality of $\psi$. As $f^* \psi_X(\rho_0) \cong (f^* F, f^* \alpha)$, it follows that $(f^* F, f^* \alpha) \cong (E, \beta)$.

    (C) Since $F'$ is also étale trivializable, the pair $(F', \alpha')$ lies in the image of $\psi_X$ by Lemma \ref{the map psi}. Thus there exists a unique homomorphism $\rho_0' \colon \pi_1^\et(X, \xi) \to \GL_n$ such that $\psi_X(\rho_0') = (F', \alpha')$. Now $(f^* F, f^* \alpha) \cong (f^* F', f^* \alpha')$ implies that $\rho_0' \circ \pi(f) = \rho_0 \circ \pi(f)$. Since $\pi(f)$ is an isomorphism, it follows that $\rho_0' = \rho_0$. Therefore, $(F, \alpha) \cong \psi_X(\rho_0) = \psi_X(\rho_0') \cong (F', \alpha')$.
\end{proof}

The following lemma represents a ``trivial case'' of our discussion.

\begin{lem}\label{lem2.3} 
    Let $E = (E_i)_{i \in \mathbb{N}}$ be an $F$-divided sheaf on $X$, and let $S \subseteq \mathbb{N}$ be an infinite subset. If the set of isomorphism classes $\{ E_i \mid i \in S \}/\cong$ is finite, then each $E_i$ is étale trivializable.
\end{lem}

\begin{proof} 
    If the set $\{ E_i \mid i \in S \} / {\cong}$ is finite, then there exists an integer $n_0$ such that for all $j \geq n_0$ with $j \in S$, we have $E_j \cong E_i$ for some $i < n_0$ with $i \in S$. Thus $E_j \cong E_i \cong (F_X^*)^{j-i} E_j$, so by Lemma \ref{Lange-Stuhler}, $E_j$ is étale trivializable. Finally, observe that if $E_j$ is étale trivializable, then so is $(F_X^*)^m E_j$ for any $m \in \mathbb{N}$. Since every $E_i$ with $i < n_0$ is a Frobenius pullback of some $E_j$ with $j \geq n_0$ and $j \in S$, the claim follows.
\end{proof}

\subsection{Local criterion of flatness} Grothendieck’s local criterion of flatness (see \cite[\S~11.3]{EGAIV3} or \cite[\href{https://stacks.math.columbia.edu/tag/00HT}{Tag 00MK}]{stacks-project} or \cite[20.E]{Mat-CA}) gives a powerful method to verify flatness by checking conditions on fibers.  
Roughly speaking, it states that a finitely generated module $M$ over a Noetherian local ring $(A,\mathfrak m)$ is flat if and only if 
\(\Tor_1^A(M, A/\mathfrak m) = 0.
\)
Equivalently, flatness can be tested after reduction modulo the maximal ideal.  
This principle underlies many geometric applications: exactness and base change can often be checked fiberwise, and the set of points where a coherent sheaf is flat is open in the base.  
We record here two consequences that will be used later.

\begin{lem}\label{local criterion of flatness} 
Let $X$ be a Noetherian scheme, and consider a sequence of vector bundles on $X$
\begin{equation}\label{exact sequence of vector bundles}
    0 \longrightarrow \sF' \xrightarrow{\;\; u \;\;} \sF \xrightarrow{\;\; w \;\;} \sF'' \longrightarrow 0.
\end{equation}
\begin{enumerate}[label={\rm(\arabic*)}]
    \item If $u$ is an isomorphism on the residue field of a point $x \in X$ 
    \textup{(i.e.\ $u \otimes_{\sO_{X,x}} \kappa(x)$ is an isomorphism)}, then $u$ is an
    isomorphism in an open neighborhood of~$x$.
    \item Suppose $f\colon Y \to X$ is a surjective morphism of schemes, with $X$ reduced.  
    If the pullback of \eqref{exact sequence of vector bundles} along $f$ is exact, then 
    \eqref{exact sequence of vector bundles} is exact.
\end{enumerate}
\end{lem}

\begin{proof}
First, assume $X = \Spec(A)$ with $A$ a local ring.  
If $u$ is an isomorphism on the residue field of the closed point of $X$, then by 
\cite[\href{https://stacks.math.columbia.edu/tag/00ME}{Tag 00ME}]{stacks-project}, 
$u$ is injective and $\Coker(u)$ is free.  
Comparing ranks shows that $\Coker(u)$ has rank zero, hence $u$ is an isomorphism.  
Since the conditions $\Ker(u)_x = \Coker(u)_x = 0$ are open in $X$, the conclusion of (1) follows for general $X$.  
For (2), note that $w \circ u = 0$ at each generic point of $X$, hence $w \circ u = 0$ globally as $X$ is reduced.  
Thus $w$ factors uniquely through a morphism $\alpha \colon \Coker(u) \to \sF''$.  
Because tensor products are right exact, $\alpha$ is an isomorphism after base change to each residue field.  
Applying (1) to $u = \alpha$, we deduce that \eqref{exact sequence of vector bundles} is right exact.  
Finally, by \cite[\href{https://stacks.math.columbia.edu/tag/00ME}{Tag 00ME}]{stacks-project}, $u$ is injective, so the sequence \eqref{exact sequence of vector bundles} is exact.
\end{proof}

\begin{lem}\label{the set of flat points is open} 
Let $f \colon X \to S$ be a proper surjective morphism of $J$-2 schemes, and let $\sF$ be a coherent sheaf on $X$ flat over $S$.  
Then the set of points
\[
U \coloneqq \{\, s \in S \mid \sF \text{ is flat over } X_s \,\}
\]
is open in $S$.
\end{lem}

\begin{proof}
If $U = \emptyset$, the claim is trivial.  
Suppose $s \in U$.  
We want to show that there exists an open neighborhood $V \subseteq S$ of $s$ such that $\sF|_{X_t}$ is flat for all $t \in V$.  

Let $Y \subseteq X$ be the closed subset where $\sF$ is not flat, and set $V \coloneqq S \setminus f(Y)$.  
By \cite[\S 11.3.10]{EGAIV3}, $\sF$ is flat at $x \in X$ if and only if $\sF|_{X_{f(x)}}$ is flat at $x \in X_{f(x)}$.  
Hence $\sF|_{X_t}$ is flat for all $t \in V$, since $\sF|_{f^{-1}(V)}$ is flat.  
This proves the claim.
\end{proof}

\subsection{Purely inseparable maps}

In positive characteristic, finite morphisms often decompose into a separable and a purely inseparable part. 
To isolate the latter, we recall Grothendieck’s notion of \emph{universally injective} (or \emph{radical}) morphisms, 
which admit several equivalent characterizations:

\begin{enumerate}
    \item For every field $K$ the induced map
        \resizebox{5.5cm}{!}{$\Hom(\Spec(K),Y)\to\Hom(\Spec(K),X)$} is injective;
    \item For any morphism $X'\to X$, the base change $f'\colon Y_{X'}\to X'$ is
        injective;
    \item The map $f$ is injective and for every $y\in Y$,
        $\kappa(y)/\kappa(f(y))$ is purely inseparable;
    \item The diagonal map $\Delta_{Y/X}\colon Y\to Y\times_XY$ is surjective.
\end{enumerate}

The following theorem of Grothendieck shows that radical morphisms are invisible to étale topology:

\begin{thm}[Grothendieck {\cite[Exposé VIII, Thm.~1.1]{SGA4}}]\label{Grothendieck} 
Let $f\colon Y\to X$ be a radical surjective integral map of
schemes. Then the base change functor $f^*\colon X_\et\to Y_\et$ induces an
equivalence of small étale sites. In particular, if $X,Y$ are
connected, and $\xi'$ is a geometric point of $Y$ with image $\xi=f(\xi')$, 
then the induced map $\pi_1^\et(Y,\xi')\to\pi_1^\et(X,\xi)$ is an
isomorphism.
\end{thm}

\noindent
A similar principle holds for $F$-divided sheaves. 
Indeed, B.~Bhatt has shown in unpublished notes that $\Fdiv(-)$ satisfies $h$-descent. 
In particular, if $f$ is a finite radical surjective morphism of
schemes of finite type over a perfect field $k$, then 
the pullback $f^*\colon \Fdiv(X)\to\Fdiv(Y)$ is an equivalence. 
For the purposes of this paper, it suffices to work with the following special case:

\begin{defn}\label{nilpotent maps} 
Suppose that $X$ is a scheme over $\F_p$. 
A morphism $f\colon Y\to X$ is called \emph{purely inseparable} if it is affine and for each open affine
$U=\Spec(A)\subseteq X$ the induced map $f^{-1}(U)=\Spec(B)\to U=\Spec(A)$ corresponds to a finite injective ring map $A\hookrightarrow B$ such
that for any $b\in B$ there exists $n\in\N^+$ with $b^{n}\in A$. 
\end{defn}

\begin{lem}\label{Setting 2.25 is radical} 
A purely inseparable map is integral, radical and surjective.
\end{lem}

\begin{proof} 
We want to prove that the map
$\Hom(\Spec(K),Y)\to\Hom(\Spec(K),X)$ is injective for every field $K$.
Equivalently, for any two maps $y,y'\in\Hom(\Spec(K),Y)$, if
their compositions with $Y\to X$ agree, then $y=y'$. 
We may assume $X=\Spec(A)$, $Y=\Spec(B)$, with $A\hookrightarrow B$
a finite injective map such that every $b\in B$ satisfies $b^n\in A$ for some $n$. 
Now in the field $K$, equality $y(b)=y'(b)$ is equivalent to $y(b)^n=y'(b)^n$ for some $n>0$, 
which holds since $y,y'$ agree on $A$. Hence $y=y'$. For the surjectivity, one just has to note that for any 
prime ideal $\mathfrak{p}\subseteq A$, $\mathfrak{p}B\neq B$ -- otherwise $1=1^{p^n}\in \mathfrak{p}$ for some $n$.
\end{proof}

\begin{prop}\label{purely inseparable maps and Fdiv} 
Let $f\colon Y\to X$ be a purely inseparable map. 
Then the pullback functor $f^*\colon \Fdiv(X)\to\Fdiv(Y)$ induces an equivalence of categories.
\end{prop}

\begin{proof} 
Since $\Fdiv(-)$ satisfies Zariski descent, it suffices to treat the affine case 
$X=\Spec(A)$, $Y=\Spec(B)$. 
For $n\gg 0$, we have a lift (dashed arrow) in the commutative diagram
\[
\xymatrix{
    Y \ar[rr]^{F_Y^n} \ar[d]_{f} && Y \ar[d]^{f} \\
    X \ar[rr]_{F_X^n} \ar@{-->}[urr]^g && X
}
\]
Both composites $f^*\circ g^*\colon\Fdiv(Y)\to\Fdiv(Y)$ and 
$g^*\circ f^*\colon\Fdiv(X)\to\Fdiv(X)$ act as ``shifts,'' hence are equivalences. 
It follows that $f^*$ itself is an equivalence.
\end{proof}

The next lemma is a relative version of the classical fact that any finite field extension 
factors into a separable and a purely inseparable part:

\begin{lem}\label{Frobenius factorization}
Let $f\colon Y\to X$ be a surjective morphism of integral schemes of finite type over a field
$k$, with $Y$ normal. 
There exists a factorization 
\[
F_Y^n\colon Y \xrightarrow{\hspace{5pt} u\hspace{5pt}} Y' \xrightarrow{\hspace{5pt}
v\hspace{5pt}} Y
\]
for some $n\in\N^+$, where $Y'$ is a normal variety separable over $X$, 
and $u,v$ are finite purely inseparable morphisms.
\end{lem}

\begin{proof} 
The map $f$ induces an extension $K(X)\subseteq K(Y)$ of function fields. 
It is elementary that the reduction of
$K(X)\otimes_{F_{K(X)}^n}K(Y)$ is a separable field extension
(cf.~\cite[\href{https://stacks.math.columbia.edu/tag/030I}{030I}]{stacks-project})
of $K(X)$ for some $n\gg 0$. 
Consider the diagram
\begin{equation}\label{Relative Frobenius}
    \begin{tikzpicture}[xscale=2.0,yscale=1.2,bmr/.pic={\draw
        (0,0)--++(-90:2mm)--++(180:2mm);},baseline={([yshift=-.5ex]current bounding box.center)}]
         \path
     (0,0)     node (F) {${F_X^n}^*Y$}
             +(-43:.4) pic[scale=1,red]{bmr}
             +(0:1.5)  node (star) {$Y$}
             ++(-90:1.5) node (X) {$X$}
             +(0:1.5)  node (Y) {$X$}
             ++(90:2.5) node (Z') {$Y'$}
             +(-180:1.) node (Z) {$Y$};
         \draw[->] (F)--(star);
         \draw[->] (Z')--(star) node[midway,above,scale=.6]{$v$};
         \draw[->] (F)--(X);
         \draw[->] (X)--(Y) node[midway,above,scale=.6]{$F_X^n$};
         \draw[->] (star)--node[midway,right,scale=.6]{$f$}(Y);
         \draw[dashed,->] (Z)--(Z')
         node[midway,above,scale=.6]{$u$};
         \draw[->] (Z')--(F)node[midway,above,scale=.3]{reduction and normalization};
         \draw[->] (Z) to [out=40,in=130]
         node[midway,above,scale=.6]{$F_Y^n$} (star);
         \draw[->] (Z) to
         [out=-90,in=140]node[midway,below,scale=.6]{$f\ $}(X);
    \end{tikzpicture}
\end{equation}
where $Y'$ is the normalization of the reduction $({F_X^n}^*Y)_\red$ in $K(\resizebox{1.45cm}{!}{$({F_X^n}^*Y)_\red$})$.
Since $({F_X^n}^*Y)_\red$ is separable over $X$, so is $Y'$. 
The existence of $u$ follows from the universal property of
normalization and the normality of $Y$. 
As $F_Y^n\colon Y\to Y$ is purely inseparable, both $u$ and $v$ are purely inseparable morphisms.
\end{proof}

\section{The case of $k=\bar{\F}_p$}\label{k the closure of Fp}
In this section we establish Theorems~\ref{ThmI} and~\ref{ThmII} in the case 
$k=\bar{\F}_p$. 
The restriction to this situation is essential, since the key technical input -- Hrushovski’s theorem -- is only available over algebraic closures of finite fields.

Let $f\colon Y\to X$ be a surjective morphism of projective varieties over
$k=\bar{\F}_p$. 
Fix a $k$-rational point $\xi'\in Y$ and set $\xi\coloneqq f(\xi')\in X$. 
For each $r\in\N^+$, the morphism $f$ induces a rational map \resizebox{8cm}{!}{
\(
f^*\colon R_X \coloneqq R(X,\xi,rp_{\sO_X}) \dashrightarrow 
R_Y \coloneqq R(Y,\xi',rp_{\sO_Y})
\)
}.
We regard $f^*$ as an actual morphism, defined on the maximal open subset of $R_X$
where the pullback is well-defined.

\subsection{Full faithfulness}
We begin with the full faithfulness property. 
The key idea is to encode morphisms between $F$-divided sheaves as sheaf $\HOM$s 
over the moduli spaces, thereby translating the problem into a question about points on these representation spaces. 
The geometry of the moduli spaces then allows us to use the known results for étale trivializable (periodic) bundles 
to approximate the corresponding results for $F$-divided sheaves.

\begin{thm}\label{full faithfulness}
    Suppose $f$ induces a surjection
    \(\pi_1^{{\rm \acute{e}t}}(Y,\xi')\twoheadrightarrow\pi_1^{{\rm
    \acute{e}t}}(X,\xi)\), and assume that $X,Y$ are smooth connected. 
    Then for any \resizebox{6.5cm}{!}{$E=(E_i,\sigma_i)_{i\in\N},
    E'=(E_i',\sigma_i')_{i\in\N}\in\Fdiv(X)$}, the pullback map
    \[
    f_{E,E'}^*\colon{\rm Hom}_{\Fdiv(X)}(E,E')\longrightarrow 
    {\rm Hom}_{\Fdiv(Y)}(f^*E,f^*E')
    \]
    is an isomorphism.
\end{thm}

\begin{proof} We'll show that $f^*_{E,E'}$ is bijective.
The pullback functor $f^*\colon \Fdiv(X)\to\Fdiv(Y)$ between Tannakian categories is automatically faithful, hence $f_{E,E'}^*$ is injective. 

For surjectivity:
 By Corollary~\ref{make F-divided sheaves points of the representation space},
 we may identify the stratified bundles $E$ and $E'$ with points of the
 representation space after an inconsequential shift. So we have
 two sequences of $k$-rational points of $R_X$:
       \(\{P_i\}_{i\in\N},\,\{Q_i\}_{i\in\N}\),
       where $P_i\coloneqq(E_i,\beta_i)$, $Q_i\coloneqq(E_i',\gamma_i)$ and
       each pullback 
       \(
       f^*P_i=(f^*E_i,f^*\beta_i)\) (resp. \(
       f^*Q_i=(f^*E_i',f^*\gamma_i)
       \))
       is also a $k$-rational point of $R_{Y,r_1}$ (resp. $R_{Y,r_2}$) with
       $r_1$ (resp. $r_2$) denoting the rank of $E$ (resp. $E'$).
       Lemma~\ref{Definition of the subscheme N} constructs the parameter space
       \resizebox{2.9cm}{!}{\(\sN\subseteq R_{X,r_1}\times_k R_{X,r_2}\)} satisfying
\begin{enumerate}[label={\rm(\arabic*)}]
       \item $(P_i,Q_i)\in\sN(k)$ for infinitely many $i\in\N$;
       \item the set $\{(P_i,Q_i)\mid i\in\N\}\cap\sN(k)$ is dense in $\sN$;
       \item $\Frob\cap\sN(k)$ is dense in $\sN$, where 
       $\Frob\subseteq R_{X,r_1}\times_k R_{X,r_2}(k)$ denotes the set of pairs $(P,Q)$
       with $P,Q$ étale trivializable.
\end{enumerate}
Let $(\sE_i^\univ,\beta_i^\univ)$ be the universal object on $X\times_kR_{X,r_i}$, and let
\(\pr_i\colon R_{X,r_1}\times_k R_{X,r_2}\to R_{X,r_i}\) ($i=1,2$) be the projections. 
Set
\[
\sE\coloneqq(\id_X\times\pr_1)^*\sE_1^\univ|_{X\times_k\sN}, 
\qquad
\sF\coloneqq(\id_X\times\pr_2)^*\sE_2^\univ|_{X\times_k\sN},
\]
and consider the maps
\[
   Y\times_k\sN\xrightarrow{f\times\id} X\times_k\sN\xrightarrow{\pr}\sN.
\]

Let $\sU\subseteq \sN$ be a dense open subset on which both sheaves
\[
\HOM(\sE,\sF),\quad
\HOM((f\times\id)^*\sE,(f\times\id)^*\sF)
\]
are locally free and commute with base change.
For any $\sQ=((V,\beta),(W,\beta'))\in\sU(k)$, the restriction of the adjunction map 
\begin{equation} \label{Ziel map}
    \pr_*\HOM(\sE,\sF)\longrightarrow\pr_*(f\times\id)_*\HOM((f\times\id)^*\sE,(f\times\id)^*\sF)
\tag{$\star$}
\end{equation}
is the canonical map
\begin{equation}\label{canonical map}
        f_{V,W}^*\colon
        \Hom(V,W)\longrightarrow\Hom(f^*V,f^*W).
\tag{$*$}
\end{equation}

\medskip\noindent
\textbf{Special case.}  
If $\sQ\in\Frob\cap\sU(k)$, i.e.\ both $V,W$ are étale trivializable, 
then the surjectivity 
\(\pi^\et_1(Y,\xi')\twoheadrightarrow\pi_1^\et(X,\xi)\)
implies that \eqref{canonical map} is an isomorphism. 
By Lemma~\ref{local criterion of flatness} (1), the locus in $\sU$ where \eqref{Ziel map} is an isomorphism is open. 
Hence, after possibly shrinking $\sU$, we may assume that \eqref{Ziel map} is an isomorphism on all of $\sU$.

\medskip\noindent
\textbf{Density argument.}  
Properties (2) and (3) are inherited by $\sU$, since $\sU\subseteq\sN$ is dense and open. 
If $\sU$ did not satisfy (1), then by (2) it would be a finite union of closed
$k$-points, forcing $\sN=\sU$, which contradicts our assumption (1) on $\sN$. 

\medskip\noindent
\textbf{Conclusion.}  
Thus for every $\sQ=((V,\beta),(W,\beta'))\in\sU(k)$, the map \eqref{canonical map} is an isomorphism. 
In particular, this holds for infinitely many pairs $(P_i,Q_i)$, i.e. there are infinitely many $i\in\N$
   such that the natural map 
   \[
       f_{E_i,F_i}^*\colon\Hom(E_i,F_i)\longrightarrow\Hom(f^*E_i,f^*F_i)
   \]
   is an isomorphism.
This shows that any morphism 
$f^*E \to f^*E'$ in $\Fdiv(Y)$ uniquely descends to a morphism $E\to E'$ in $\Fdiv(X)$.
Thus $f_{E,E'}^*$ is surjective, completing the proof.
\end{proof}

\subsection{Essential surjectivity} We now prove essential surjectivity. Given an $F$-divided sheaf $(E_i', \sigma_i')$ on $Y$, the naive strategy is to use Corollary~\ref{make F-divided sheaves points of the representation space} to realize the $E_i'$ as points on a moduli space $R_Y$ and then lift them to $X$ via the rational pullback $f^*: R_X \dashrightarrow R_Y$.

The core difficulty is that while individual bundles $E_i'$ can often be lifted, assembling these lifts into a coherent $F$-divided sheaf is obstructed. This would be tractable if the fibers of $f^*$ were finite, but this fails in general. For instance, if $f$ is a Nori-reduced $G$-torsor for a finite connected $k$-group scheme $G$, then by \cite[Lemma 2.5]{TZ3}, the fiber over $\mathcal{O}_Y^{\oplus r}$ is in bijection with $r$-dimensional representations of $G$, which is infinite.

Our approach is to first treat the case where $f$ is separable, avoiding this issue, and then use Lemma \ref{Frobenius factorization} to handle the general case. The main complication is that the factorization process temporarily breaks smoothness, requiring additional techniques to handle singular varieties.

\subsubsection{The separable case} In the separable case, the moduli space can
be refined to ensure the pullback map $f^*$ has single fibers, effectively rigidifying the lifting problem and allowing for a canonical choice of lift.

Recall that for $r\in\N^+$, we denote $R_{X,r}\coloneqq R(X,\xi,rp_{\sO_X})$.
For an $n$-tuple $\vec{r}\coloneqq(r_1,\dots,r_n)$, set $R_{X,\vec{r}}\coloneqq R_{X,r_1}\times\cdots\times R_{X,r_n}$.

\begin{lem}\label{object}
   Suppose that $f$ is separable, surjective, and induces an isomorphism on the
   étale fundamental groups.
   Let
   $\Sigma\coloneqq\{Q_i'\}_{i\in\N}\subseteq R_{Y,\vec{r}}(k)$ be a sequence of rational
   points with $F_Y^*Q_{i+1}'=Q_i'$. Then there
   exists an integral locally closed subscheme $\sU'\subseteq R_{Y,\vec{r}}(k)$ such that:
   \begin{enumerate}[label={\rm(\arabic*)}]
       \item $\sU'(k)$ contains $Q_i'$ for infinitely many $i\in\N$;
       \item $\Sigma\cap \sU'(k)$ is dense in $\sU'$;
       \item $\Frob_Y\cap\sU'(k)$ is dense, where
           \[
              \Frob_Y\coloneqq\Set{Q'=(Q^{(1)},\dots,Q^{(n)})\in
              R_{Y,\vec{r}}(k)\mid Q^{(i)}\ \text{is étale trivializable $\forall\, i$}}
           \]
       \item the morphism ${f^*}^{-1}(\sU')\to\sU'$ is finite, purely inseparable, and surjective. 
   \end{enumerate}
\end{lem}

\begin{proof}
Let $\sN'\subseteq R_{Y,\vec{r}}(k)$ be the closed integral subscheme obtained via
Lemma~\ref{Definition of the subscheme N}, and set $\sN\coloneqq {f^*}^{-1}(\sN')$.  
By Corollary~\ref{cor3.2}(B), the periodic points of $\sN'$ are mapped onto,
so $f^*:\sN \to \sN'$ is dominant. 
Moreover, by Corollary~\ref{cor3.2}(C) and Lemma~\ref{separable fiber unique}, for any $Q' \in \Frob_Y$, the fiber ${f^*}^{-1}(Q')$ consists of a single point.

By Chevalley’s theorem
(cf.~\cite[\href{https://stacks.math.columbia.edu/tag/05F9}{05F9}, \href{https://stacks.math.columbia.edu/tag/054E}{054E}]{stacks-project}),
the locus $\sN_0'\subseteq\sN'$ where the fibers have dimension~$0$ is constructible.  
Since $\Frob_Y\subseteq\sN_0'$, the subset $\sN_0'$ is dense in $\sN'$.  
In view of
\cite[\href{https://stacks.math.columbia.edu/tag/005K}{005K}]{stacks-project},
there exists a dense open $\sU'\subseteq\sN'$ (contained in $\sN_0'$) such that
$\sU\coloneqq{f^*}^{-1}(\sU')\to \sU'$ is quasi-finite.  

Because $\Frob_Y\cap\sN'$ is dense in $\sN'$, the intersection $\Frob_Y\cap\sU'$ is dense in $\sU'$.  
Thus, there is a dense subset of $\sU'(k)$ whose fibers along $\sU\to\sU'$ are singletons.  

Next, we show that $\sU\to\sU'$ can be assumed finite after shrinking $\sU'$.  
If some generic points of $\sU$ do not map to the generic point of $\sU'$, we may remove the closure of their images.  
Suppose there are distinct generic points $\eta_1,\eta_2\in\sU$ mapping to the same generic point of $\sU'$.  
We can find disjoint open neighborhoods $U_i$ of $\eta_i$ ($i=1,2$).  
But then, by
\cite[\href{https://stacks.math.columbia.edu/tag/005K}{005K}]{stacks-project},
the intersection $\im(U_1)\cap\im(U_2)$ contains a dense open of $\sU'$.  
In particular, it contains some $Q\in\Frob_Y$ whose fiber under $\sU\to\sU'$ is a singleton, a contradiction to $U_1\cap U_2=\emptyset$.  
Hence, after shrinking $\sU'$, we may assume $\sU$ is irreducible.  

Since nilpotent thickenings of affine schemes are affine
(cf.~\cite[\href{https://stacks.math.columbia.edu/tag/05YU}{05YU}]{stacks-project}),
$\sU\to\sU'$ is finite iff $\sU_\red\to\sU'$ is finite.  
Thus we may further assume that $\sU$ is reduced.  
By
\cite[Ch.~III, Ex.~3.7]{Hartshorne77}, shrinking $\sU'$ if necessary, 
we conclude that $\sU\to\sU'$ is finite.  

Finally, as $\sU\to\sU'$ is closed and dominant, it is surjective.  
Since there is a dense subset of $\sU'(k)$ whose fibers are singletons, the
induced function field extension \resizebox{2.5cm}{!}{$K(\sU')\subseteq
K(\sU_\red)$} must be purely inseparable.  
Thus $\sU\to\sU'$ is finite, purely inseparable, and surjective, as required.  
\end{proof}


The above lemma ensures the existence of a canonical lift $E_i$ of $E_i'$. 
However, this lift is ``canonical'' only in the sense of being determined as a 
point of the moduli space. In general, there may exist other lifts outside the 
moduli space. For instance, $F_X^*E_{i+1}$ is also a lift of $E_i$. If it is 
only $\mu$-semistable but not $p$-semistable, then it does not correspond to a 
point of the moduli space, and hence we cannot conclude that 
$F_X^*E_{i+1}\cong E_i$. This issue disappears when each $E_i'$ is 
$\mu$-stable, since $\mu$-stability descends along surjective morphisms and
$\mu$-stability implies $p$-stability.

\begin{lem}\label{stable bundles}
   Let $f\colon Y\to X$ be a separable surjective morphism of normal varieties
   inducing an isomorphism on étale fundamental groups:
   $\pi^\et_1(Y,\xi')\xrightarrow{\cong} \pi_1^\et(X,\xi)$. 
   Suppose that $Y$ is smooth.
   Then for any $(E_i',\sigma_i')_{i\in\N}\in\Fdiv(Y)$ with $E_i'$ being $\mu$-stable, there exists an $F$-divided sheaf
   $(E_i,\sigma_i)_{i\in\N}\in\Fdiv(X)$ such that 
   {\rm (1)} $f^*E_i\cong E_i'$ for all $i$;
   {\rm(2)} each $E_i$ is again $\mu$-stable with vanishing Chern classes.
\end{lem}

\begin{proof} The smoothness of $Y$ guarantees (cf.~Corollary \ref{make
    F-divided sheaves points of the representation space}) the existence of a collection $\Sigma \coloneqq \{Q_i' = (E_i', \beta_i')\}_{i \in \mathbb{N}}$ of $k$-points in the moduli space $R_Y$ satisfying:
\begin{enumerate}[label=(\arabic*)]
    \item $F_Y^* Q_{i+1}' = Q_i'$ for all $i \in \mathbb{N}$;
    \item Each bundle $E_i'$ is $\mu$-stable with vanishing Chern classes.
\end{enumerate}

    By Lemma~\ref{object}, there exists a locally closed integral subscheme
    $\sU'\subseteq R_Y$ such that $\sU'(k)$ contains $\{Q_i'\}_{i\in S}$ for an infinite
    subset $S\subseteq\N$, and ${f^*}^{-1}(\sU')(k)\to \sU'(k)$ is a bijection. 
    For each $i\in S$, let $Q_i=(E_i,\beta_i)$ denote the unique $k$-point of 
    ${f^*}^{-1}(\sU')$ mapping to $Q_i'$.
    If $j>i$ with $j\in S$, then 
    $f^*({F_X^{j-i}}^*Q_j)\cong Q_i'$. 
    Since $E_i'$ is $\mu$-stable and $f$ is surjective, 
    ${F_X^{j-i}}^*E_j$ is also $\mu$-stable. 
    Hence ${F_X^{j-i}}^*Q_j\in {f^*}^{-1}(\sU')(k)$. 
    Because ${f^*}^{-1}(\sU')(k)\to \sU'(k)$ is a bijection, it follows that 
    ${F_X^{j-i}}^*Q_j=Q_i$. 

    Therefore, the sequence $\{E_i\}_{i\in S}$ extends uniquely to an
    $F$-divided sheaf $(E_i,\sigma_i)\in\Fdiv(X)$ satisfying $f^*E_i\cong E_i'$. 
    Each $E_i$ is $\mu$-stable, and the vanishing of their Chern classes follows 
    from the projection formula for Chern classes 
    (cf.~\cite[\href{https://stacks.math.columbia.edu/tag/02U9}{02U9}]{stacks-project}).
\end{proof}

Finally, we approximate general $F$-divided sheaves by $\mu$-stable ones through extensions.

\begin{lem}\label{separable for general bundles}
   Let $f: Y \to X$ be a separable surjective morphism of smooth varieties
   inducing an isomorphism on étale fundamental groups:
   $\pi^\et_1(Y,\xi')\xrightarrow{\cong} \pi_1^\et(X,\xi)$.  Given an
    $F$-divided sheaf $E'=(E_i',\sigma_i)_{i\in\N}\in\Fdiv(Y)$, there exists an $F$-divided sheaf
            $E\coloneqq(E_i,\sigma_i)_{i\in\N}\in\Fdiv(X)$ with $f^*E\cong E'$.\end{lem}
\begin{proof} 
The smoothness of $Y$ guarantees (cf.~Theorem~\ref{EM10 main lemma}) the existence of a filtration whose graded pieces are of the form
    $U'=(U_i',\theta_i')_{i\in\N}$ with: {\rm(1)} each
    $U_i'$ is $\mu$-stable; {\rm(2)} all Chern class vanish.
    By Corollary \ref{make F-divided sheaves points of the representation
    space}, we get a collection $\Sigma\coloneqq\{Q_i'\coloneqq (E_i',\beta_i')\}_{i\in\N}$ of
            $k$-points of the moduli space
            $R_Y$ satisfying $F_Y^*Q_{i+1}'=Q_i'$. 

            Following the argument in
            Lemma \ref{stable bundles}, we proceed as follows. Let $Q_i\coloneqq
    (E_i,\beta_i)$ be the unique $k$-point of ${f^*}^{-1}\sU'$ mapping to $Q_i'$
    for all $i\in S$.
    If $j>i$ and $j\in S$, then $f^*{F_X^{j-i}}^*Q_j\cong Q_i'$. If ${F_X^{j-i}}^*Q_j\in
    R_X(k)$, \emph{i.e. $f^*{F_X^{j-i}}^*E_i$ is semistable}, then it must lie in ${f^*}^{-1}\sU'(k)$. The fact that ${f^*}^{-1}\sU'(k)\to
    \sU'(k)$ is a bijection implies that $f^*{F_X^{j-i}}^*Q_j=Q_i$. This allows
    us to extend the sequence $\{E_i\}_{i\in S}$ to an
    $F$-divided sheaf  $(E_i,\sigma_i)\in\Fdiv(X)$ satisfying $f^*E_i\cong
    E_i'$. Then one can apply \cite[Prop. 1.7]{Gi75} to conclude.

    To show that
${F_X^{j-i}}^*Q_j\in
R_X(k)$, we first assume that $E'$ is an extension of
$U'=(U_i',\theta_i')_{i\in\N}$ by $V'=(V_i',\theta_i')_{i\in\N}$, where each
$U_i', V_i'$ is stable. Lemma \ref{object} implies that there are
$U=(U_i,\theta_i)_{i\in\N}$ and $V=(V_i',\theta_i)_{i\in\N}$ satisfying
$f^*U_i\cong U_i'$ and $f^*V_i\cong V_i'$. Moreover, we have a moduli space
$\sU'\subseteq R_{Y,r_1}\times R_{Y,r_2}\times R_{Y,r_3}$ satisfying 
\begin{enumerate}[label={(\arabic*)}]
    \item $\sU'(k)$ contains \resizebox{3.6cm}{!}{$((U_i',\beta_{U_i'}),Q_i',
        (V_i',\beta_{V_i'}))$} for infinitely many $i\in\N$;
    \item the set of points
        \resizebox{4cm}{!}{$\left\{((U_i',\beta_{U_i'}),Q_i',
        (V_i',\beta_{V_i'}))\right\}$} is dense in $\sU'$;
    \item $\Frob_Y\cap\sU'(k)$ is dense, where
        \begin{center}
        \resizebox{0.91\textwidth}{!}{$
              \Frob_Y\coloneqq\Set{P'=(P^{(1)},P^{(2)},P^{(3)})\in
              R_{Y,r_1}\times R_{Y,r_2}\times R_{Y,r_3}(k)\mid P^{(i)}\ \text{is étale trivializable $\forall\, i$}}
      $}\end{center}
    \item the morphism ${f^*}^{-1}(\sU')\to\sU'$ is finite, purely inseparable, and surjective.  
\end{enumerate}
In light of the proof of Theorem \ref{full faithfulness}, we can shrink
$\sU\coloneqq {f^*}^{-1}(\sU')$ further so that each exact sequence $0\to V_i'\to E_i'\to U_i'\to 0$ 
 descend to a sequence
 \begin{equation}\label{extension by stable bundles}
0\longrightarrow V_i\longrightarrow E_i\longrightarrow U_i\longrightarrow 0
 \end{equation} for infinitely many $i$.

Indeed, let $(\sE_i^\univ,\beta_i^\univ)$ be the universal object on $X\times_kR_{X,r_i}$, and let
\(\pr_i\colon R_{X,r_1}\times_k R_{X,r_2}\times_k R_{X,r_3}\to R_{X,r_i}\)
($i=1,2,3$) be the projections. 
Set
\begin{center}
\resizebox{0.9\textwidth}{!}{\(
\sF_1\coloneqq(\id_X\times\pr_1)^*\sE_1^\univ|_{X\times_k\sU},\, 
\sF_2\coloneqq(\id_X\times\pr_2)^*\sE_2^\univ|_{X\times_k\sU},\,
\sF_3\coloneqq(\id_X\times\pr_3)^*\sE_3^\univ|_{X\times_k\sU},\,
\)}\end{center}
and consider the maps
\(
   Y\times_k\sU\xrightarrow{f\times\id} X\times_k\sU\xrightarrow{\pr}\sU.
\)

Shrinking $\sU$ if necessary, we may assume that all sheaves
\begin{center}
\resizebox{0.9\textwidth}{!}{\(
\HOM(\sF_1,\sF_2),\,\HOM((f\times\id)^*\sF_1,(f\times\id)^*\sF_2),\,
\HOM(\sF_2,\sF_3),\,
\HOM((f\times\id)^*\sF_2,(f\times\id)^*\sF_3)
\)}\end{center}
are locally free and commute with base change.
For any $\sQ=((V,\beta_V),(E,\beta_E),(U,\beta_U))\in\sU(k)$, the restriction
of the adjunction maps 
\begin{align} 
    \pr_*\HOM(\sF_1,\sF_2)\longrightarrow\pr_*(f\times\id)_*\HOM((f\times\id)^*\sF_1,(f\times\id)^*\sF_2) \label{Ziel map VE}\\
    \pr_*\HOM(\sF_2,\sF_3)\longrightarrow\pr_*(f\times\id)_*\HOM((f\times\id)^*\sF_2,(f\times\id)^*\sF_3) \label{Ziel map EU}
\end{align}
are the canonical maps
\begin{align}\label{canonical map VE}
        f_{V,E}^*\colon
        \Hom(V,E)\longrightarrow\Hom(f^*V,f^*E)\\
f_{E,U}^*\colon
        \Hom(E,U)\longrightarrow\Hom(f^*E,f^*U)
\end{align}  
If $\sQ\in\Frob\cap\sU(k)$, i.e. all $V,E,U$ are étale trivializable, 
then the surjectivity of
\(\pi^\et_1(Y,\xi')\twoheadrightarrow\pi_1^\et(X,\xi)\)
implies that \eqref{Ziel map VE} and \eqref{Ziel map EU} are isomorphisms. 
Since by Lemma~\ref{local criterion of flatness} (1), the locus in $\sU$ where
\eqref{Ziel map VE} and \eqref{Ziel map EU} are isomorphisms is open, after possibly shrinking $\sU$, we may assume that \eqref{Ziel map VE} and \eqref{Ziel map EU} are isomorphisms on all of $\sU$.
Note that even after shrinking $\sU$, we still have that  
$\sU(k)$ contains $\sQ_i=((V_i,\beta_{V_i}),Q_i,(U_i,\beta_{U_i}))$ for infinitely many
$i\in\N$. In other words, for infinitely many points
$\sQ_i\in\sU(k)$, \eqref{Ziel map VE} and
\eqref{Ziel map EU} are isomorphisms. This is exactly what we wanted.

 Thanks to Lemma \ref{local criterion of flatness}, the sequence
 \eqref{extension by stable bundles} is exact. It follows from Lemma
 \ref{stable bundles} that ${F_X^{j-i}}^*U_j\cong U_i$ and
 ${F_X^{j-i}}^*V_j\cong V_i$, so
 ${F_X^{j-i}}^*E_j$ is $p$-semistable with reduced Hilbert polynomial
 $p_{\sO_X}$. Thus ${F_X^{j-i}}^*Q_j\in
R_X(k)$.  The general case then follows from induction on the rank of $E'$.
\end{proof}
\subsubsection{The general case} 
We now assemble the separable and purely inseparable cases. 
Because the preceding lemmas were carefully formulated to avoid smoothness assumptions, 
this final step of the argument becomes considerably simpler.
\begin{thm}\label{separable for general bundles final version}
   Let $f: Y \to X$ be a surjective morphism of smooth varieties over
   $k=\bar{\F}_p$, inducing an isomorphism of étale fundamental groups:
   $\pi^{\et}_1(Y,\xi')\xrightarrow{\cong} \pi_1^{\et}(X,\xi)$. Then
   for any $F$-divided sheaf $E' \in \Fdiv(Y)$, there exists an $F$-divided
   sheaf $E \in \Fdiv(X)$ with an isomorphism: $F^*E\cong E'$ in $\Fdiv(Y)$.\end{thm}
\begin{proof} Thanks to Corollary \ref{make F-divided sheaves points of the
    representation space}, up to an inconsequential shift, $E'$ is a successive extension of objects in $\Fdiv(Y)$ of the form $U'=(U_i',\theta_i')_{i\in\N}$, where each
   $U_i'$ is $\mu$-stable with reduced Hilbert polynomial
   $p_{\sO_Y}$. 

   In view of Lemma \ref{Frobenius factorization} we have a factorization
   $F_Y^n\colon Y\xrightarrow{u} Y'\xrightarrow{v} Y$, where $Y'$ is normal and
   separable over $X$ and $u,v$ are finite purely inseparable. By
   Theorem \ref{Grothendieck}, $u$ induces an isomorphism on $\pi_1^\et$.
   According to Lemma \ref{purely inseparable maps and Fdiv}, $E'$ together
   with the successive extension descend uniquely along $u$ to $Y'$.
   Moreover, if we denote $V'=(V_i',\tau_i')_{i\in\N}\in\Fdiv(Y')$ the descent
   of $U'$, then each $V_i'$ has vanishing Chern class  by the projection formula of
   Chern classes (cf.~\cite[\href{https://stacks.math.columbia.edu/tag/02U9}{02U9}]{stacks-project}). Thus each $V_i'$ is $\mu$-stable with reduced Hilbert polynomial
   $p_{\sO_{Y'}}$. Replacing $Y$ by $Y'$ we can assume that $Y$ is separable over $X$. However, $Y$ is only normal instead of being smooth.

   We now apply Lemma \ref{separable for general bundles} to conclude the argument. Although $Y$ may not be smooth, the $F$-divided sheaf $E'$ remains a successive extension of $\mu$-stable objects. Consequently, each $E_i'$ is $p$-semistable and can therefore be realized as a point in the representation space. This ensures that the argument of Lemma \ref{separable for general bundles} remains valid in this setting.
\end{proof}

\section{For general $k$}

In this section, we extend the results of the previous sections to an arbitrary perfect base field $k$, 
using the technique of ``spreading out''. The idea is as follows: Given a surjective map 
$f\colon Y\to X$ of smooth projective $k$-schemes inducing an isomorphism on the étale fundamental group, 
we choose a finite type $\F_p$-algebra $A \subseteq k$ such that $f$ admits a model 
$f_S\colon Y_S \to X_S$, where $f_S$ is a morphism of smooth projective schemes over 
$S = \Spec(A)$. Moreover, $f_S$ can be chosen so that for each geometric point $s$ of $S$, 
the geometric fiber $f_s$ induces an isomorphism of the étale fundamental groups. 
Since the geometric points of $S$ can be taken to be $\bar{\F}_p$, this construction allows 
us to approximate the generic fiber $f$ using the known results for the
geometric fiber $f_s$.

Compared with the ``spreading out'' employed in \cite{EM10}, the main difficulty here 
is that the model $f_S$ must be constructed carefully to preserve the isomorphism of étale 
fundamental groups on the closed fibers. In the classical case, any smooth
projective model $X_S/S$ suffices, 
because the specialization map is surjective, so triviality of the fundamental
group of the generic fiber 
implies the triviality of the special fiber.

In this section, we fix a proper surjective morphism $f\colon Y\to X$ between geometrically integral schemes 
of finite type over a field $k$ of characteristic $p$.
 
\subsection{Model for $f$ separable}
Suppose that $f$ is separable. Our first objective is to establish that such morphisms admit a spreading-out property.

\begin{lem}\label{fiberwise separable}
    Let $f_S\colon Y_S\to X_S$ be a morphism between flat proper geometrically integral
    schemes over a Noetherian integral base scheme $S$ with fraction field
    $K\subseteq k$. If the base change $f_S\times_S k = f$ is separable, then there exists a nonempty open subset
    $U\subseteq S$ such that for every geometric point $s$ of $U$, the fiber map 
    $f_s\colon Y_s\to X_s$ is separable.   
\end{lem}

\begin{proof}
    Since $S$ is integral and $X_S, Y_S$ are flat and geometrically integral over $S$, 
    both are integral schemes. By assumption, the generic fiber 
    $f_S\times_S K$ is separable, hence $f_S$ is generically separable. 
    
    Let $V\subseteq Y_S$ denote the locus where $f_S$ is smooth. By generic smoothness, 
    $V$ is a nonempty open subset of $Y_S$. Its image $U\coloneqq f_S(V)$ is constructible 
    by Chevalley’s theorem, and since it contains the generic point of $S$, it must be open and nonempty. 
    
    For any geometric point $s\in U$, the fiber $Y_s$ meets $V$, so $f_s$ is smooth at some point of $Y_s$, 
    which implies that $f_s$ is separable. 
\end{proof}
 
We now present the technical core of this section: we will demonstrate that the property of ``inducing an isomorphism on étale fundamental groups of geometric fibers'' is a spreading-out property. Since our subsequent approach involves factoring general morphisms into a purely inseparable morphism followed by a separable one, we must develop our theory in the context of possibly singular varieties.

\begin{lem}\label{fiberwise pi-iso} 
    Let $f_S \colon Y_S \to X_S$ be a map between proper geometrically normal schemes {\rm(cf.~\cite[\href{https://stacks.math.columbia.edu/tag/038Z}{038Z}]{stacks-project})} over a Noetherian integral scheme $S$ with fraction field $K \subseteq k$. Assume:
    \begin{enumerate}[label={\rm(\arabic*)}]
        \item $f_S \times_S k = f$,
        \item $X_S$ is smooth over $S$, and
        \item $\pi_1^{\et}(Y_{\bar{k}}) \to \pi_1^{\et}(X_{\bar{k}})$ is an isomorphism.
    \end{enumerate}
    Then there exists a nonempty open subset $U \subseteq S$ such that for every geometric point $s$ of $U$, the induced map $\pi_1^{\et}(Y_s) \to \pi_1^{\et}(X_s)$ is an isomorphism.
\end{lem}
\begin{proof}     Thanks to 
    Lemma \ref{fiberwise separable}, we can assume that $f_s$ is
    separable for each geometric point $s$ in $S$. In view of \cite[Exposé X, Thm
    3.8]{SGA1},  for each $s$, the map $\pi_1^\et(Y_{s})\to
    \pi_1^\et(X_{s})$ is surjective. Consequently, the
    pullback of finite étale covers $f_s^*\colon\FEt(X_s)\to\FEt(Y_s)$ is fully
    faithful. To prove essential surjectivity,
    consider $Y'\to Y_s\in \FEt(Y_s)$. Using
    \cite[\href{https://stacks.math.columbia.edu/tag/054F}{054F}]{stacks-project}
    then performing strict henselization
    \cite[\href{https://stacks.math.columbia.edu/tag/07QL}{07QL}]{stacks-project}
    and completion
    \cite[\href{https://stacks.math.columbia.edu/tag/07NU}{07NU}]{stacks-project}, we construct:
    \begin{itemize}
        \item A complete DVR $R$ with algebraically closed residue field; 
        \item A map
    $\Spec(R)\to S$ mapping the special point to $s$, the generic point
    to $\eta_S$ (the generic point of $S$).
    \end{itemize}
     By Grothendieck's existence theorem,
    $Y'\to Y_s$ extends uniquely to $Y_R'\to Y_R\in\FEt(Y_R)$. Set $L\coloneqq
    \Frac(R)$. The pullback $Y_{\bar{L}}'\to
    Y_{\bar{L}}\in\FEt(Y_{\bar{L}})$ descends, via the isomorphism
    $\pi_1^\et(Y_{\bar{L}})\to \pi_1^\et(X_{\bar{L}})$, to a unique cover
    $X_{\bar{L}}'\to X_{\bar{L}}\in\FEt(X_{\bar{L}})$. 

    There exists a finite intermediate extension $L\subseteq
    L_1\subseteq\bar{L}$ and a cover
    $X_{L_1}'\to X_{L_1}\in\FEt(X_{{L_1}})$ whose base change to $\bar{L}$ is
    $X_{\bar{L}}'\to X_{\bar{L}}$. Let $R_1$ be the integral closure of $R$
    in $L_1$, giving a finite extension of complete DVRs with
    identical residue field. 

    The relative normalization
    $X_{R_1}'$ of $X_{R_1}$ in $X_{L_1}'$    fits in a  
     diagram (cf.~\cite[\href{https://stacks.math.columbia.edu/tag/0BAK}{0BAK}]{stacks-project}) 
    \begin{equation}\label{relative normalization}
                \begin{tikzpicture}[xscale=2.0,yscale=1.2,baseline={([yshift=-.5ex]current bounding box.center)}]
                     \path
                     (0,0)     node (F) {$Y_{R_1}'$}
                         +(0:1.5)  node (star) {$X_{R_1}'$}
                         ++(-90:1.5) node (X) {$Y_{R_1}$}
                         +(0:1.5)  node (Y) {$X_{R_1}$};
                     \draw[->] (F)--(star);
                     \draw[->] (F)--(X) node[midway,left,scale=.6]{$\lambda_Y$};
                     \draw[->] (X)--(Y) node[midway,above,scale=.6]{$f_{R_1}$};
                     \draw[->] (star)--(Y) node[midway,right,scale=.6]{$\lambda_X$};
                \end{tikzpicture}
    \end{equation}
    \textbf{Key observations: }\begin{itemize}
        \item The normalization map $\lambda_X$ is finite   by
    \cite[\href{https://stacks.math.columbia.edu/tag/0AVK}{0AVK}]{stacks-project});
        \item The diagram is cartesian over $L_1$ (since it is so over
            $\bar{L}$);
        \item For
    $\eta_s\in Y_{R_1}$ the generic point of $Y_s$, generic smoothness implies
    $f_R$ is smooth on some open neighborhood
    $V\ni\eta_s$;
        \item By
    \cite[\href{https://stacks.math.columbia.edu/tag/03GC}{03GC}]{stacks-project},
    $\lambda_Y^{-1}(V)=V\times_{X_{R_1}}X_{R_1}'$, so $\lambda_X$ is étale at 
    $f_s(\eta_s)$.
\end{itemize}
    Since $X_{R_1}'$ is normal
    (cf.~\cite[\href{https://stacks.math.columbia.edu/tag/035L}{035L}]{stacks-project})
    and $X_{R_1}$ is regular, Zariski-Nagata purity
    (cf.~\cite[\href{https://stacks.math.columbia.edu/tag/0BMB}{0BMB}]{stacks-project})
    shows $\lambda_X$ is étale. Thus the diagram \eqref{relative normalization} is
    cartesian, and $Y'\to Y_s$ is indeed the pullback  of
    the special fiber
    of $\lambda_X$.
\end{proof}

\subsection{Model for general $f$} 

Now we are ready to assemble the ingredients. Thanks to \cite{stacks-project}, 
the following result is essentially a direct consequence of 
\href{https://stacks.math.columbia.edu/tag/03YE}{the descent of properties of morphisms}, 
together with the standard factorization trick, which decomposes a morphism $f$ 
into its purely inseparable part followed by its separable part.

\begin{thm}\label{descending properties} Let $f\colon Y\to X$ be a proper surjective morphism between
geometrically integral schemes proper over a
field $k$ of characteristic $p$. There exists a smooth affine $\F_p$-scheme
$S=\Spec(A)$ with an
    inclusion $A\subseteq k$, and an
    $S$-morphism $f_S\colon X_S\to Y_S$ such that
    $f_S\times_Sk=f$, and after possibly modifying $S$, the following
    properties hold: 
    \begin{enumerate}[label={\rm (\arabic*)}]
        \item $X_S,Y_S$ are integral schemes;
        \item $X_S,Y_S$ are proper, flat, surjective, and geometrically
            reduced geometrically connected over $S$;
        \item if $f$ is affine, then $f_S$ is affine;
        \item if $f$ is finite, then $f_S$ is
            finite;
        \item if $X,Y$ are projective over $k$, then $X_S,Y_S$ are
            $S$-projective;
        \item if $X,Y$ are smooth {\lp resp. separable and geometrically normal {\rm
            \cite[\href{https://stacks.math.columbia.edu/tag/038L}{038L}]{stacks-project}}\rp} over $k$,  then $X_S,Y_S$ are
            $S$-smooth {\lp resp. $S$-separable and $S$-normal
            {\rm\cite[\href{https://stacks.math.columbia.edu/tag/038Z}{038Z}]{stacks-project}}\rp};
        \item if
            $\pi_1^{\et}(Y)\to\pi_1^{\et}(X)$ is surjective, then
            $\pi_1^{\et}(Y_s)\to\pi_1^{\et}(X_s)$ is surjective for all the
            geometric fibers $s$ of $S$;
        \item if $X,Y$ are smooth over $k$ and if
            $\pi_1^{\et}(Y)\to\pi_1^{\et}(X)$ is bijective, then
            $\pi_1^{\et}(Y_s)\to\pi_1^{\et}(X_s)$ is bijective for all the
            geometric fibers $s$ of $S$.
    \end{enumerate}
\end{thm}
\begin{proof}
    (1)-(6) follow from
    \cite[\href{https://stacks.math.columbia.edu/tag/07SK}{07SK},
    \href{https://stacks.math.columbia.edu/tag/084V}{084V},
    \href{https://stacks.math.columbia.edu/tag/081A}{081A}]{stacks-project}.
    (7) follows readily from the specialization of the étale
    fundamental group \cite[Exposé X, Théorème 3.8]{SGA1} or \cite[\href{https://stacks.math.columbia.edu/tag/0BUQ}{0BUQ}]{stacks-project}.

    For (8), we analyze the function field extension
    $K(X)\subseteq K(Y)$ through the following steps. Let $K$ be an
    intermediate extension such that $K/K(X)$ is separable and $K(Y)/K$ is purely inseparable.
    Let $Z$ be the normalization of $X$ in $K$, giving a factorization: 
    \[
        Y\xrightarrow{\quad g\quad}Z\xrightarrow{\quad h\quad} X
    \]
    where $h$ is surjective separable and $g$ is finite surjective purely
    inseparable.
    After shrinking $S$, we obtain compatible models: $X_S,Y_S$ smooth proper
    geometrically connected; $Z_S$ proper $S$-normal; $h_S$ surjective
    separable; $g_S$
    finite surjective purely inseparable. For every
    geometric points $s$ of $S$, Theorem \ref{Grothendieck} implies
    that $\pi_1^{\et}(Z_s)\to\pi_1^{\et}(X_s)$ is an isomorphism, and  Lemma \ref{fiberwise
    pi-iso} gives that $\pi_1^{\et}(Y_s)\to\pi_1^{\et}(Z_s)$ is an isomorphism. Putting them together we conclude the proof.
\end{proof}

\subsection{Reconstructing the moduli}

After the spreading-out preparation, the arguments from \S \ref{k the closure of Fp} 
can be repeated almost verbatim, except that we now work with families of varieties 
(schemes over $S$) rather than with a single variety. Nevertheless, since $S$ is of finite type over $\F_p$, 
any such family can itself be regarded as a variety over $\F_p$.

Let $X$ be a connected projective variety over a field
$k=\bar{k}$ equipped with a $k$-rational point $\xi$. We choose a smooth
    $\F_p$-algbera $A\subseteq k$ so that $(X,\xi)$ admits a geometrically connected
    projective model $X_S\to S\coloneqq\Spec(A)$ and $\xi_S\colon S\to X_S$.
    Set $R_{X_S}\coloneqq R(X_S,\xi_S,rp_{\sO_{X}})$. Viewing
    $\Spec(k)$ as
    an $S$-scheme via $A\subseteq k$, we have
    $R_{X_S}\times_Sk=R_X$. For $r\in\N^+$, we denote $R_{X_S,r}\coloneqq
    R(X_S,\xi,rp_{\sO_{X_S}})$.
For an $n$-tuple $\vec{r}\coloneqq(r_1,\dots,r_n)$, set $
R_{X_S,\vec{r}}\coloneqq R_{X_S,r_1}\times\cdots\times R_{X_S,r_n}$.

\begin{lem}\label{Definition of the subscheme N, relative case}  Let $\Sigma\coloneqq\{(Q_i^{(1)},\cdots,Q_i^{(n)})\}_{i\in\N}$ be a collection of
            $k$-points of the moduli space
            $R_{X_S,\vec{r}}$ such that
            $F_X^*Q_{i+1}^{(1)}=Q_i^{(1)}$, $\cdots$, $F_X^*Q_{i+1}^{(n)}=Q_i^{(n)}$. There is a closed integral subscheme
            $\sN_S\subseteq R_{X_S,\vec{r}}$ such that 
    \begin{enumerate}[label=\textup{(\Alph*)}]
        \item $\sN_S(k)$ contains $(Q_i^{(1)},\cdots,Q_i^{(n)})$ for infinitely many $i\in\N$;
        \item $\Sigma\cap\sN_S(k)$ is dense in $\sN_S$; 
        \item $\sN_S$ contains a dense subset of $\bar{\F}_p$-points of the form
            $(Q_i^{(1)},\cdots,Q_i^{(n)})$  with
            $Q_i^{(1)}$, $\cdots$, $Q_i^{(n)}$ being periodic in $\sN_S(\bar{\F}_p)$.
    \end{enumerate}
\end{lem}
\begin{proof} By Lemma \ref{Definition of the subscheme N}, we find a closed
    integral
    subscheme $\sN\subseteq R_{X_S,\vec{r}}$ verifying the conditions
    (A), (B), (C) of that Lemma. After modifying $S$, we find a closed
    integral subscheme $\sN_S$ together with a rational map $(F_{X_S}^*)^a\colon
            \sN_S\dashrightarrow \sN_S$ which serve as models of $\sN$ and
            $(F_{X}^*)^a$. Result (A) of Lemma \ref{Definition of the
            subscheme N} implies result (A) of this lemma.

            Since $\Sigma\cap \sN(k)$ is dense in $\sN$, its closure in $\sN_S$
            contains $\sN$ which contains the generic point of $\sN_S$,
            consequently, $\Sigma\cap
            \sN(k)=\Sigma\cap \sN_S(k)$ is dense in $\sN_S$, whence (B). This
            also shows that the rational map $(F_{X_S}^*)^a$ is dominant.
            Finally, (C) follows from \cite[Lemma 3.6]{Sun19}.
\end{proof} 

\subsection{Full faithfulness}
Let $f\colon Y\to X$ be a map of     connected
    projective varieties over an algebraically closed field $k$. We choose a smooth
    $\F_p$-algbera $A\subseteq k$,  so that $(X,\xi)$ (resp. $(Y,\xi')$)  admit a geometrically
    smooth connected
    projective model $(X_S,\xi_S)$ (resp. $(Y_S,\xi_S')$), and $f$ lifts to
    $f_S\colon Y_S\to X_S$,
    where $S\coloneqq \Spec(A)$. The map $f_S$ induces a
rational map $f_S^*\colon R_{X_S}\dashrightarrow
R_{Y_S}$, which will be regarded as an actual map by considering its maximal
open subset of definition.

\begin{lem}\label{Morphisms relative case}
   Suppose that $f$ induces a surjective map
   \(\pi^\et_1(Y,\xi')\twoheadrightarrow\pi_1^\et(X,\xi)\). Consider the
   following data:
   \begin{itemize}
       \item two
   sequences of $k$-rational points of $R_X$:
   $\{P_i\}_{i\in\N},\,\{Q_i\}_{i\in\N}$,
   where $P_i\coloneqq(E_i,\beta_i)$, $Q_i\coloneqq(F_i,\gamma_i)$ and
   $E_i$ \lp resp. $F_i$\rp is a rank $r_1$ \lp resp. $r_2$\rp vector bundle;
       \item  
   suppose that every $f^*P_i=(f^*E_i,f^*\beta_i)$ and
        $f^*Q_i=(f^*F_i,f^*\gamma_i)$ is a $k$-rational points in $R_Y$.
\end{itemize}
Suppose that there is a locally closed integral subscheme $\sN_S\subseteq
R_{X_S,r_1}\times_S R_{X_S,r_2}$ 
\begin{enumerate}[label={\rm(\arabic*)}]
    \item whose
$k$-rational points contain pairs $(P_i,Q_i)$ for infinitely many $i\in\N$;
    \item $\Set{(P_i,Q_i)|i\in\N}\cap\sN_S(k)$ is dense in $\sN_S$; 
    \item 
        $\Frob\cap\sN_S(\bar{\F}_p)$ is dense in $\sN_S$, where $\Frob\subseteq
        R_{X_S,r_1}(\bar{\F}_p)\times R_{X_S,r_2}(\bar{\F}_p)$ denotes the subset of $\bar{\F}_p$-points of the form $(P,Q)$ with $P,Q$
 being étale trivializable.
\end{enumerate}
Then there are infinitely many $i\in\N$
   such that the natural map 
   \[
       f_{E_i,F_i}^*\colon\Hom(E_i,F_i)\longrightarrow\Hom(f^*E_i,f^*F_i)
   \]
   is an isomorphism.
\end{lem}

\begin{proof} The proof here is analogous to that of Theorem \ref{full
    faithfulness}. 
    Let $(\sE^\univ_i,\beta^\univ_i)$ be the universal object on
    $X_S\times_SR_{X_S,r_i}$, and let
\(
\pr_i\colon R_{X_S,r_1} \times_S R_{X_S,r_2}\to R_{X_S,r_i}
\) ($i=1,2$),
be the projections. We consider sheaves:
$\sE\coloneqq(\id_{X_S}\times\pr_1)^*\sE_1^\univ|_{X_S\times_S\sN_{S}}$ and
$\sF\coloneqq(\id_{X_S}\times\pr_2)^*\sE_2^\univ|_{X_S\times_S\sN_{S}}$; maps
\(
   Y_S\times_S\sN_S\xrightarrow{f_S\times\id_S} X_S\times_S\sN_S\xrightarrow{\pr}\sN_S
\).

Let $\sU_S\subseteq \sN_S$ be a dense open subset where 
\begin{itemize}
    \item both sheaves
\[\pr_*\HOM(\sE,\sF),\text{
}\pr_*(f\times\id)_*\HOM((f\times\id)^*\sE,(f\times\id)^*\sF)\]
commute with base change;
\item the adjunction map
\begin{equation} \label{Ziel map}
    \pr_*\HOM(\sE,\sF)\longrightarrow\pr_*(f\times\id)_*\HOM((f\times\id)^*\sE,(f\times\id)^*\sF)
\tag{$\star$}
\end{equation} is an isomorphism.
\end{itemize}
When $\sQ\coloneqq ((V,\beta),
 (W,\beta'))\in\Frob\cap\sU_S(\bar{\F}_p)$, 
    i.e. the bundles $V,W$ are étale trivializable, the surjectivity of \(\pi^\et_1(Y,\xi')\twoheadrightarrow\pi_1^\et(X,\xi)\)
    and Lemma \ref{descending properties} (7) imply that  the canonical map
 \begin{equation}\label{canonical map}
        f_{V,W}^*\colon
        \Hom(V,W)\longrightarrow\Hom(f^*V,f^*W)\tag{$*$}
    \end{equation}   which is \eqref{Ziel map} restricting to $\sQ$, is
    an isomorphism.
Lemma \ref{local criterion of flatness} (1) guarantees that $\sU_S\neq\emptyset$. 

 Properties (2), (3) on $\sN_S$ are inherited by $\sU_S$ since
 $\sU_S\subseteq \sN_S$ is dense open. For any subset $T\subseteq
 \sU_S(k)$, if $T$ is not dense in $\sU_S\times_Sk$, then its closure there can not
 contain the generic point, hence $T$ is not dense in $\sU_S$. If $\sU_S$ fails (1), then, by (2),
 $\sU_S\times_Sk$ must
 be
 a finite union of (closed) $k$-points implying $\sU_S(k)$ is a singleton.
 Then $\sN_S\times_Sk$ is an integral $k$-variety containing a dense open
 $\sU\times_Sk\cong\Spec(k)$. This implies that
 $\sN_S\times_Sk=\sU\times_Sk$ which
 contradicts our assumption that $\sN_S$ satisfies (1).

 \textbf{Conclusion}: for every point $\sQ\coloneqq ((V,\beta),
 (W,\beta'))\in\sU_S(k)$, the canonical map
\eqref{canonical map} is the pullback
of \eqref{Ziel map}, hence is itself an isomorphism.
\end{proof}

\begin{thm}\label{full faithfulness the general case} 
    Suppose $f$ induces a surjection
$\pi_1^{\NN,{\rm \acute{e}t}}(Y,\xi')\twoheadrightarrow\pi_1^{\NN,{\rm
\acute{e}t}}(X,\xi)$.  Assume that $X,Y$ are smooth and $k$ is perfect.
Then for any
$E=(E_i,\sigma_i)_{i\in\N},E'=(E_i',\sigma_i')_{i\in\N}\in\Fdiv(X)$, the
pullback map
\[f_{E,E'}^*:{\rm Hom}_{\Fdiv(X)}(E,E')\longrightarrow {\rm
Hom}_{\Fdiv(Y)}(f^*E,f^*E')\] is an isomorphism.
\end{thm}
\begin{proof}
By Proposition \ref{Base change for Nori} and \cite[Appendix, Prop.~B3]{TZ2}, we may assume that $k=\bar{k}$, so $\pi_1^{\NN,\et}=\pi_1^\et$.

    Replacing Lemma \ref{Definition of the subscheme N} with Lemma \ref{Definition of the subscheme N,
    relative case} and using Lemma \ref{Morphisms
    relative case} in the proof of Theorem \ref{full faithfulness}, we immediately obtain the desired result. 
\end{proof}

\subsection{Essential surjectivity} 
\begin{lem}\label{object the general case}
   Suppose that $f_S\colon Y_S\to X_S$ is separable, surjective and induces an isomorphism
   $\pi^\et_1(Y_s,\xi_s')\xrightarrow{\cong} \pi_1^\et(X_s,\xi_s)$ for all
   geometric points $s$ of $S$. Let
   $\Sigma\coloneqq\{Q_i'\}_{i\in\N}\subseteq R_{Y_S,\vec{r}}$ be a sequence of rational points with $F_Y^*Q_{i+1}'=Q_i'$.  There
   is an integral locally closed subscheme $\sU'_S\subseteq R_{Y_S,\vec{r}}$ such that
   \begin{enumerate}[label={\rm(\arabic*)}]
       \item $\sU'_S(k)$ contains $Q_i'$ for infinitely many $i\in\N$;
       \item $\Sigma\cap \sU'_S(k)$ is dense in $\sU'_S$;
       \item $\Frob_Y\cap\sU'_S(\bar{\F}_p)$ is dense, where \[\Frob_Y\coloneqq\Set{Q'=(Q^{(1)},\cdots,Q^{(n)})\in
           R_{Y_S,\vec{r}}(\bar{\F}_p)|Q^{(i)}\text{ is étale trivializable}}\]
       \item ${f_S^*}^{-1}(\sU'_S)\to\sU'_S$ is finite purely inseparable and surjective. 
   \end{enumerate}
\end{lem}
\begin{proof} Let $\sN_S'\subseteq R_{Y_S,\vec{r}}$ be the closed integral subscheme
    obtained via Lemma \ref{Definition of the subscheme N, relative case}, and
    let $\sN_S\coloneqq {f_S^*}^{-1}(\sN_S')$.   By Corollary \ref{cor3.2} (B), the
    periodic $\bar{\F}_p$-points of $\sN_S'$ are mapped onto, so $f_S^*:\sN_S \to \sN_S'$
is a dominant map. 
By Corollary \ref{cor3.2} (C) and Lemma \ref{separable fiber unique}, for any
$\sQ\in\Frob_Y$, the fiber ${f_S^*}^{-1}(\sQ)$ is a singleton as a set. The fact
that $\sN_S'$ contains a dense subset of points whose fibers are singletons
allows us to run the same argument employed in the proof of Lemma \ref{object}:
\begin{itemize}
    \item Applying a
theorem of Chevalley
(cf.~\cite[\href{https://stacks.math.columbia.edu/tag/05F9}{05F9},
\href{https://stacks.math.columbia.edu/tag/054E}{054E}]{stacks-project}), we
find a dense open $\sU_S'\subseteq \sN_S'$ such that
$\sU_S\coloneqq {f_S^*}^{-1}(\sU_S')\to\sU_S'$ is quasi-finite;
    \item Applying \cite[Ch. III, Exercise 3.7]{Hartshorne77} we can shrink
        $\sU_S'$ so that
        $f_S|_{\sU_S}$ becomes a finite map between integral schemes.
\end{itemize}
Finally, one concludes that $f_S|_{\sU_S}$ is finite surjective and
        purely inseparable.
\end{proof}

\begin{lem}\label{stable bundles, general case}
    Suppose  $f_S\colon Y_S\to X_S$ satisfies: {\rm(1)} it is separable surjective;
    {\rm(2)} $X_S$ is $S$-normal and $Y_S$ is $S$-smooth; {\rm (3)} it
   induces an isomorphism on étale fundamental groups:
   $\pi^\et_1(Y_s,\xi_s')\xrightarrow{\cong} \pi_1^\et(X_s,\xi_s)$ for all
   geometric point
   $s$. For any $(E_i',\sigma_i')_{i\in\N}\in\Fdiv(Y)$ with $E_i'$ being $\mu$-stable, there exists an $F$-divided sheaf
            $(E_i,\sigma_i)_{i\in\N}\in\Fdiv(X)$ such that {\rm\lp 1\rp}
            $f^*E_i\cong
            E_i'$;
            {\rm\lp 2\rp} each $E_i$ is again $\mu$-stable with vanishing Chern class.
\end{lem}
\begin{proof}
The smoothness of $Y$ guarantees (cf.~Corollary \ref{make
    F-divided sheaves points of the representation space}) the existence of a collection $\Sigma\coloneqq\{Q_i'\coloneqq (E_i',\beta_i')\}_{i\in\N}$ of
            $k$-points of the moduli space
            $R_{Y_S}$ satisfying: {\rm\lp 1\rp} $F_{Y}^*Q_{i+1}'=Q_i'$; {\rm\lp 2\rp} each $E_i'$ is  $\mu$-stable
              with vanishing Chern class.

    Replacing Lemma \ref{object} with Lemma \ref{object the general case}, the
    proof is almost identical to the proof of Lemma \ref{stable bundles}.
\end{proof}

\begin{lem}\label{separable for general bundles, general case}
    Let $f: Y \to X$ be a separable surjective morphism where: {\rm (1)} $X,Y$
    are
    smooth;  {\rm (2)} $f$ induces an isomorphism
    $\pi^\et_1(Y,\xi')\xrightarrow{\cong} \pi_1^\et(X,\xi)$. Given an
    $F$-divided sheaf $E'=(E_i',\sigma_i)_{i\in\N}\in\Fdiv(Y)$, there exists an $F$-divided sheaf
            $E\coloneqq(E_i,\sigma_i)_{i\in\N}\in\Fdiv(X)$ with $f^*E\cong E'$.
\end{lem}
\begin{proof} The proof here is essentially the same as that of Lemma \ref{separable for general bundles}:
\begin{itemize}
    \item The smoothness of $Y$ guarantees (cf.~Theorem~\ref{EM10 main lemma}) the existence of a filtration whose graded pieces are of the form
    $U'=(U_i',\theta_i')_{i\in\N}$ with: {\rm(1)} each
    $U_i'$ is $\mu$-stable; {\rm(2)} all Chern class vanish.
    By Corollary \ref{make F-divided sheaves points of the representation
    space}, we get a collection $\Sigma\coloneqq\{Q_i'\coloneqq (E_i',\beta_i')\}_{i\in\N}$ of
            $k$-points of the moduli space
            $R_Y$ satisfying $F_Y^*Q_{i+1}'=Q_i'$.

    \item Using Lemma \ref{object the general case}, we find a sequence of sheaves $\{E_i\}_{i\in S}$ satisfying $f^*E_i\cong
    E_i'$;
    \item Using Lemma \ref{stable bundles, general case}, the stable quotients from the filtration descend to $F$-divided sheaves on $X$;
    \item By Lemma \ref{Morphisms relative case} and Lemma \ref{local criterion
        of flatness}, each $E_i$ is a successive extenison of $\mu$-stable bundles whose Frobenius pullbacks are still $\mu$-stable, so the Frobenius pullbacks of $E_i$ are $p$-semistable;
    \item The sequence of bundles thus form an $F$-divided sheaf $E\coloneqq(E_i,\sigma_i)_{i\in\N}$ by the uniqueness of the lift via Lemma \ref{object the general case}.
    \item Finally, we conclude $f^*E=E'$ via \cite[Prop. 1.7]{Gi75}.
\end{itemize}
\end{proof}

\begin{thm}\label{separable for general bundles over general fields final version}
   Let $f: Y \to X$ be a surjective morphism of smooth varieties over a
   perfect field, inducing an isomorphism of Nori-étale fundamental groups:
   $\pi^{\NN,\et}_1(Y,\xi')\xrightarrow{\cong} \pi_1^{\NN,\et}(X,\xi)$. Then
   for any $F$-divided sheaf $E' \in \Fdiv(Y)$, there exists an $F$-divided
   sheaf $E \in \Fdiv(X)$ with an isomorphism: $F^*E\cong E'$ in $\Fdiv(Y)$.
\end{thm}
\begin{proof} By Proposition \ref{Base change for Nori} and \cite[Appendix, Prop.~B3]{TZ2}, we may assume that $k=\bar{k}$, so $\pi_1^{\NN,\et}=\pi_1^\et$.
    The proof is then the same as that of Theorem \ref{separable for general bundles final version}:
    \begin{itemize}
        \item First write $E'$ (or its some shift) as a successive extension of $F$-divided sheaves with stable components;
        \item Using Lemma \ref{Frobenius factorization} reduce the problem to the case when $f$ is separable surjective. But pay attention to the fact that this reduction removes the smoothness assumption on $Y$ and replaced it by the normality;
        \item Finally, we apply the proof of Lemma \ref{separable for general bundles, general case} to conclude.
    \end{itemize}
\end{proof}

\section*{Acknowledgements} 
We would like to thank Hélène Esnault, Xiaowen Hu, Adrian Langer, Vasudevan
Srinivas, Lei
Zhang (USTC) and Kang Zuo for very helpful discussions.

\printbibliography


\end{document}